\documentclass{amsart}
\usepackage{amssymb,amsmath,amscd, amsthm, color,graphicx,bbm,stackrel,paralist,eufrak,comment,hyperref}
\usepackage[font=small,labelfont=bf]{caption}
\usepackage[color,matrix,arrow]{xy}
\usepackage{enumerate}
\usepackage{xcolor}

\newtheorem{lemma}{Lemma}[section]
\newtheorem{thm}[lemma]{Theorem}

\newtheorem{cor}[lemma]{Corollary}
\newtheorem{defn}[lemma]{Definition}
\newtheorem{thm*}{Theorem}

\theoremstyle{definition}

\theoremstyle{remark}

\numberwithin{equation}{section}

\renewcommand{\bar}{\overline}

\newcommand{\T}{\mathbf{T}}
\renewcommand{\S}{\mathbf{S}}
\newcommand{\Z}{\mathbb{Z}}

\renewcommand{\v}{\mathbf{v}}
\newcommand{\w}{\mathbf{w}}

\renewcommand{\P}{\mathcal{P}}

\newcommand{\so}{\Rightarrow}
\newcommand{\os}{\Leftarrow}
\renewcommand{\iff}{\Leftrightarrow}

\newcommand{\N}{\mathbb{N}}

\newcommand{\x}{\mathbf{x}}

\newcommand{\g}{\mathbf{g}}
\newcommand{\e}{\mathbf{e}}
\newcommand{\0}{\mathbf{0}}

\renewcommand{\r}{\mathbf{r}}

\newcommand{\ep}{\epsilon}

\renewcommand{\i}{\mathbf{i}}

\newcommand{\actson}{\curvearrowright}
\newcommand{\gothG}{\mathfrak{G}}
\newcommand{\gothH}{\mathfrak{H}}

\newcommand{\F}{\mathcal{F}}
\newcommand{\FF}{\mathcal{FF}}

\begin{document}

\title[Finite odometer factors of rank one $\Z^d$-actions]{Finite odometer factors of rank one $\Z^d$-actions}

\author[A.S.A. Johnson]{Aimee S. A. Johnson$^1$} 

\thanks{$^1$Corresponding author:  Department of Mathematics and Statistics, Swarthmore College, 500 College Ave.
Swarthmore, PA 19081; email:  \texttt{aimee@swarthmore.edu}}

\author[D. McClendon]{ and David M. McClendon$^2$}

\thanks{$^2$Department of Mathematics, Ferris State University, ASC 2021, Big Rapids, MI 49307; email:  \texttt{mcclend2@ferris.edu}}



\date{June 15, 2023}

\maketitle

\begin{abstract}
In this paper, we give explicit conditions characterizing the F{\o}lner rank one $\Z^d$-actions that factor onto a finite odometer; those that factor onto an arbitrary, but specified $\Z^d$-odometer, and those that factor onto an unspecified $\Z^d$-odometer.  We also give explicit conditions describing the F{\o}lner rank one $\Z^d$-actions that are conjugate to a specific $\Z^d$-odometer, and those that are conjugate to some $\Z^d$-odometer.  These conditions are based on cutting and stacking procedures used to generate the action, and generalize results given in \cite{FGHSW} for rank one $\Z$-actions.
\end{abstract}

%
%

\section{Introduction}

In recent work, Foreman, Gao, Hill, Silva and Weiss \cite{FGHSW} gave a characterization of rank one transformations which factor onto an odometer (equivalently, they characterize which rank one transformations are not totally ergodic) in terms of parameters coming from a cutting and stacking construction used to describe the rank one transformation.  They used this characterization to describe the rank one transformations which are conjugate to a given odometer, and the rank one transformations that are conjugate to some odometer.  

The work in \cite{FGHSW} can be thought of as one step towards a more ambitious goal of classifying rank one actions up to conjugacy.  In particular, the conjugacy relation on rank one systems is known to be a Borel equivalence relation \cite{FRW}, and so there should be an explicit way of describing all rank one transformations that are conjugate to a specific rank one transformation.  Foreman et al. address exactly this question in the setting where the specific rank one transformation is an odometer.  

In this paper, we generalize the work in \cite{FGHSW} to F{\o}lner rank one actions of $\Z^d$ (though we believe our results hold for actions of groups beyond $\Z^d$).  We characterize, in Theorem \ref{finitefactorthm}, those actions which factor onto a specified finite $\Z^d$-odometer.  Using this result, we then characterize the F{\o}lner rank one $\Z^d$-actions that factor onto a specific odometer (Corollary \ref{rankoneodometerfactor}) and those that factor onto an arbitrary infinite odometer (Theorem \ref{someodometerfactor}).  Finally, we describe the F{\o}lner rank one $\Z^d$-actions that are
conjugate to a specific $\Z^d$-odometer (Theorem \ref{conj2odom}) and those that are conjugate to some $\Z^d$-odometer (Theorem \ref{conj2someodom}).  We also show, via an example in Theorem \ref{stackingnonex}, that the results of \cite{FGHSW} do not necessarily generalize to rank one $\Z^d$-actions when the tower shapes do not form a F{\o}lner sequence in $\Z^d$.  In Theorem \ref{examplenottotallyergodic}, we show that,
 unlike the situation for rank one $\Z$-actions, for F{\o}lner rank one $\Z^d$-actions total ergodicity does not coincide with the existence of a finite odometer factor.  Finally, in Theorem  \ref{examplesubactions} , we give an example of a F{\o}lner rank one $\Z^2$-action generated by transformations that each factor onto a free odometer but yield a $\Z^2$-action that does not itself factor onto a free odometer.

When $d=1$, our results can also be seen as a generalization of the results of \cite{FGHSW} to funny rank one $\Z$-actions whose tower shapes form a F{\o}lner sequence; this would include, for example, funny rank one $\Z$-actions constructed by Ferenczi \cite{Fe2} that are known not to be loosely Bernoulli and thus not of rank one.

\emph{Acknowledgment:}  The authors wish to acknowledge the Little School Dynamics research group of the American Institute of Mathematics, through which they were initially exposed to the work in \cite{FGHSW}.

%
%

\section{Rank one actions and odometers}


\subsection{Dynamical systems}

We begin with some standard definitions from measurable dynamics.  Let $G$ be a group and let $(X,\mathcal{X},\mu)$ be a Lebesgue probability space.  We say $(X,\mu,\T)$ is a \textbf{measure-preserving $G$-action}, and write $\T : G \actson (X,\mu)$, if for every $g \in G$ there is an invertible, measurable, measure-preserving transformation $\T^g: X \to X$, and these transformations satisfy $\T^{gh} = \T^g \circ \T^h$ for every $g, h \in G$ and also that $\T^{0}(x) = x$ for all $x \in X$, where $0$  denotes the identity element of $G$.

While the measure of a Lebesgue probability space is by definition non-atomic, in this paper we will also want to discuss systems with finite phase spaces.  With this in mind, if $X$ is a finite set, $\mathcal{X}$ is the power set of $X$, and $\mu$ is normalized counting measure on $(X,\mathcal{X})$, we will also decree any action of group $G$ on the finite set $X$ that preserves $\mu$ to be a \textbf{measure-preserving $G$-action}.

A measure-preserving $G$-action $(X,\T)$ is called \textbf{free} if, for almost every $x \in X$, $\T^g(x) = x$ implies $g = 0$.  If the only measurable sets invariant under every $\T^g$ have $\mu$-measure $0$ or $1$, then we say $(X,\mu,\T)$ is \textbf{ergodic}.  

We say that one measure-preserving $G$-action $(X,\mu,\T)$ \textbf{factors} onto another $(Y,\nu,\S)$ if there is a measurable function $\phi : X \to Y$ so that $\phi^*(\mu) = \nu$ and $\phi(\T^g(x)) = \S^g(\phi(x))$ for almost every $x \in X$ and every $g \in G$.  Any such $\phi$ is called a \textbf{factor map}.  Two actions $\T : G \actson (X,\mu)$ and $\S : G \actson (Y,\nu)$ are \textbf{conjugate} if there is a factor map from $(X,\T)$ to $(Y,\S)$ which is almost surely one-to-one.

In this paper, we concern ourselves with actions where $G = \Z^d$ for some $d$, but our results likely translate to actions of more general discrete amenable groups.


\subsection{Rank one systems}

Rank one actions form a rich class of measure-preserving actions that provide examples, and especially counterexamples, addressing many questions of measurable dynamics.  They have their roots in Chacon's ``geometric construction'' of examples of systems with particular mixing properties \cite{Cha1, Cha2} and in an example of a measure-preserving transformation with no square root due to Ornstein \cite{Orn}.  The name ``rank one'' was coined in \cite{ORW}, and many properties of rank one actions are well-known:  for instance, they are ergodic \cite{Fr, RS}, have simple spectrum \cite{Bax, RS} and are of zero entropy \cite{Bax, RS}. 

The study of rank one systems, as with the rest of ergodic theory, spread from actions generated by a single transformation to actions of more general groups.  Rank one $\Z^d$-actions were first considered by Johnson and {\c{S}}ahin in \cite{JS}, where the authors showed that any such action was loosely Bernoulli.  Their work built on the one-dimensional work of \cite{ORW} and the generalization of Feldman's $\bar{f}$-metric to $\Z^d$-actions found in \cite{H}.
 More recently, Danilenko has described rank one actions of more general groups using $(C,F)$-models \cite{Dan1, Dan2}, although these actions might better be described as generalizing Thouvenot's notion of a \emph{funny rank one} $\Z$-action (see \cite{Fe2}).

Today there are many different approaches to defining rank one systems which are ``mostly'' equivalent (see \cite{Fe, AFP} for details), including a constructive geometric definition coming from Chacon's work, a constructive symbolic definition \cite{Kal, dJR} and a non-constructive symbolic definition \cite{dJ}.  Our description of rank one systems is what is formally called the non-constructive geometric approach and informally referred to as \emph{cutting and stacking}.  This approach has its origin in Ornstein's example and is rooted in the notion of periodic approximation as studied in \cite{KS}.   The following definitions can be found in \cite{RS} but are repeated here for completion.
 
To begin, we define a \textbf{shape} to be a nonempty finite subset of $\Z^d$.  Then a sequence $\{F_n\}$ of shapes is called a \textbf{F{\o}lner sequence} for $\Z^d$ if for any $\v \in \Z^d$, 
\[
\lim_{n \to \infty} \frac{\#\left(F_n \, \triangle \, [F_n + \v]\right)}{\#(F_n)}  = 0.
\]
A shape can be associated to a measure-preserving $\Z^d$-action by partitioning the space into subsets, arranged via the shape in the following way:

\begin{defn}
Let $\T: \Z^d \actson (X,\mu)$ and let $R\subset \Z^d$ be a shape.  If there exists a measurable set $B$ such that $\mu(B)>0$ and $\T^{\v_1}(B) \cap \T^{\v_2}(B) = \emptyset$ for every $\v_1, \v_2 \in R$, then the $\T$-{\bf{tower of shape $R$}} is defined to be $\mathcal{T} = \{ \T^{\v}(B) : \v \in R\} \cup E$, where $E = \left( \cup_{\v\in R}\T^{\v}(B) \right)^C$.  The sets $\T^{\v}(B)$ are then the {\bf{levels}} of the $\T$- tower; the set $B$ is the {\bf{base}} of the $\T$-tower; and $E$ is called the \textbf{\emph{error set}} of $\mathcal{T}$.
\end{defn}

\noindent Note that we can assume $\0\in R$ by simply shifting the indices of $R$.  

If $\mathcal{T}$ is a $\T$-tower and $x \in X$, then we write $x \in \mathcal{T}$ to mean that $x$ is not in the error set of $\mathcal{T}$.

We are now ready to define what it means for a $\Z^d$-action to be rank one:

\begin{defn} 
\label{rankonedefn}
Let $\T: \Z^d \actson (X,\mu)$ and let $\{F_n\}$ be a sequence of shapes.  We say $\T$ is \textbf{\emph{rank one for $\{F_n\}$}} if there is a sequence $\{\mathcal{T}_n\}$ of $\T$-towers of shape $F_n$, so that given any $A \in \mathcal{X}$, there is a sequence of  $\mathcal{T}_n$-measurable sets $L_n$ (i.e. a union of levels of the tower $\mathcal{T}_n$) so that
\[\lim_{n \to \infty} \mu\left(L_n \, \triangle \, A\right) = 0.\]
We say $\T$ is \textbf{\emph{rank one}} if it is rank one for some sequence $\{F_n\}$ of shapes.
\end{defn}

In Definition \ref{rankonedefn}, we make no assumptions about the sequence of shapes.  The following definitions are more restrictive.

\begin{defn} A rank one $\Z^d$-action $\T: \Z^d \actson (X,\mu)$ is called $\{F_n\}$-\textbf{\emph{F{\o}lner rank one}}  if $\T$ is rank one for  $\{F_n\}$ where $\{F_n\}$ is a F{\o}lner sequence for $\Z^d$.  If $\{F_n\}$ is understood, we simply say $\T$ is \textbf{\emph{F{\o}lner rank one}}.
\end{defn}

\begin{defn} A rank one $\Z^d$-action $\T: \Z^d \actson (X,\mu)$ is called  \textbf{\emph{stacking rank one}} if $\T$ is rank one for  $\{F_n\}$ and the associated $\T$-towers $\mathcal{T}_n$ are such that every $\mathcal{T}_{n+1}$ refines $\mathcal{T}_{n}$ and the error sets decrease: $E_{n+1} \subseteq E_n$ for all $n$. \end{defn}

In other words, every level in $\mathcal{T}_{n}$ is a union of levels from $\mathcal{T}_{n+1}$, and thus one can think of a stacking rank one system as being one associated with a ``cutting and stacking" construction.  The two types of rank one systems listed above are related, as described in the following theorem.

\begin{thm}[\cite{RS}, Theorem 3.8] Let $\{F_n\}$ be a F{\o}lner sequence for $\Z^d$ with $\mathbf{0} \in F_n$ for all $n$.  Then any $\{F_n\}$-F{\o}lner rank one $\Z^d$-action is stacking rank one, with towers whose shapes are given by $\{F_n\}$.
\end{thm}

Now consider a $\Z^d$-action which is F{\o}lner rank one, and thus is also stacking.  Fix $m$ and consider the $\T$-tower $\mathcal{T}_m$ of shape $F_m$.  Because $\T$ is stacking, we know the base $B_m$ of $\mathcal{T}_m$ is refined by the $\T$-tower $\mathcal{T}_n$ of shape $F_n$ for any $n\ge m$.  That is, for each $n\ge m$, there is a union of levels in $\mathcal{T}_n$ which equal $B_m$, leading us to the following definition:

\begin{defn} Given a F{\o}lner rank one $\Z^d$-action $\T: \Z^d \actson (X,\mu)$ and any $n, m \in \mathbb{N}$ with  $n\ge m$, define
\[I_{m,n} = \{\i \in F_n : \T^\i(B_n) \subseteq B_m \}.\]
In \cite{FGHSW}, these sets are called the {\bf{descendants}} of the base $B_m$.
\end{defn}

One of the reasons to restrict to $\Z^d$-actions that are F{\o}lner rank one is given by the following lemma, which identifies a relationship between any F{\o}lner sequence and any finite index subgroup of $\Z^d$.  More specifically, it says that eventually, every set in the F{\o}lner sequence intersects each coset of a finite index subgroup in roughly the same proportion.

%
%

\begin{lemma}
\label{folnerlemma}
 Let $\{F_n\}$ be a F{\o}lner sequence in $\Z^d$, and let $G$ be a finite-index subgroup of $\Z^d$.  Then for every $\v\in\Z^d$, 
 \[ 
 \stackbin[n \to \infty]{}{\lim} \frac{\#(F_n \cap (\v + G))}{\#(F_n)} = \frac 1{[\Z^d:G]}.
 \]
\end{lemma}

\begin{proof}
Suppose not.  Then there is a coset $\w + G$ in $\Z^d/G$ and a subsequence $\{F_{n_k}\}$ of $\{F_n\}$ 
so that one of two things happens:  either there is a constant $\alpha \in (0,1)$ so that for all sufficiently large $k$,
\begin{equation}
\label{folnerequation1}
  \frac{\#(F_{n_k} \cap (\w + G))}{\#(F_{n_k})} \leq \frac \alpha{[\Z^d:G]},
\end{equation}
or there is a constant $\beta > 1$ so that for all large enough $k$, 
\begin{equation}
\label{folnerequation1a}
  \frac{\#(F_{n_k} \cap (\w + G))}{\#(F_{n_k})} \geq \frac \beta{[\Z^d:G]}.
\end{equation}


In the first situation, notice that it is also the case that for every $k$, there exists a coset $\v + G$ such that
\[
\#\left(F_{n_k} \cap (\v + G)\right) \geq \frac{1}{[\Z^d : G]} \, \#(F_{n_k}).
\]
Since there are only finitely many cosets in $\Z^d/G$, it follows that there is a single coset $\v + G \in \Z^d/G$ and a subsequence 
$\{F_{n_{k_l}}\}$ of $\{F_{n_k}\}$ so that for all $l$, 
\begin{equation}
\label{folnerequation2}
\frac{\#\left(F_{n_{k_l}} \cap (\v + G)\right)}{\#(F_{n_{k_l}})} \geq \frac{1}{ [\Z^d : G]}.
\end{equation}

Now consider $\x \in F_{n_{k_l}} \cap (\v + G)$.  In particular, this says $\x \in (\v + G)$, and so we have $\x - \v \in G$ 
and $\x - \v + \w \in \w + G$.  
Either $\x - \v + \w $ is in $F_{n_{k_l}}$ or it is not: if it is, then $\x - \v + \w \in F_{n_{k_l}} \cap (\w + G)$; if it is not, then 
 $\x - \v + \w \in F_{n_{k_l}} \, \triangle \, \left[F_{n_{k_l}} + (\w - \v)\right]$.  We thus have
 \begin{equation}
 \label{triangleinequalityone}
 \# \left(\, F_{n_{k_l}} \cap (\v + G)\, \right) \le \# \left( F_{n_{k_l}} \, \triangle \, \left[F_{n_{k_l}} + (\w - \v)\right] \right) \, + 
 \, \# \left(  F_{n_{k_l}} \cap (\w + G)  \right).
\end{equation}
Using (\ref{folnerequation1}) on the right and (\ref{folnerequation2}) on the left, we get
\[
 \frac{1}{ [\Z^d : G]}\,  \# ( F_{n_{k_l}}  ) \le \# \left( F_{n_{k_l}} \, \triangle \, \left[F_{n_{k_l}} + (\w - \v)\right] \right) \, +
 \frac \alpha{[\Z^d:G]}\,  \#(F_{n_{k_l}} ),
\]
 yielding that for all $l$ we have
 \[
  \frac{1-\alpha}{ [\Z^d : G]}\, \#(F_{n_{k_l}} ) \le\# \left( F_{n_{k_l}} \, \triangle \, \left[F_{n_{k_l}} + (\w - \v)\right] \right)  
 \]
 or 
 \[
   \frac{1-\alpha}{ [\Z^d : G]} \le \frac{ \# \left( F_{n_{k_l}} \, \triangle \, \left[F_{n_{k_l}} + (\w - \v)\right] \right)}{\#(F_{n_{k_l}} )}.
 \]

We thus have that
\[
\lim_{l\rightarrow\infty} \frac{ \# \left( F_{n_{k_l}} \, \triangle \, \left[F_{n_{k_l}} + (\w - \v)\right] \right)}{\#(F_{n_{k_l}} )} \neq 0
\]
which contradicts $\{F_n\}$ being a F{\o}lner sequence.


A similar argument works in the second situation:  this time, we notice that for every $k$, there is $\v + G$ so that
\[\#(F_{n_k} \cap (\v + G)) \leq \frac 1{[\Z^d : G]} \#(F_{n_k}),\]
so there is a single coset $\v + G$ and a subsequence $\{F_{n_{k_l}}\}$ of $F_{n_k}$ so that for all $l$,
\begin{equation}
\label{folnerequation3}
\frac{\#\left(F_{n_{k_l}} \cap (\v + G)\right)}{\#(F_{n_{k_l}})} \leq \frac{1}{ [\Z^d : G]}.
\end{equation}
Repeating the argument used to obtain (\ref{triangleinequalityone}) but with the roles of $\v$ and $\w$ interchanged, we obtain
\[
 \# \left(\, F_{n_{k_l}} \cap (\w + G)\, \right) \le \# \left( F_{n_{k_l}} \, \triangle \, \left[F_{n_{k_l}} + (\v - \w)\right] \right) \, + 
 \, \# \left(  F_{n_{k_l}} \cap (\v + G)  \right).
\]
Now using (\ref{folnerequation1a}) on the left and (\ref{folnerequation3}) on the right we get
\[
 \frac{\beta}{ [\Z^d : G]}\,  \# ( F_{n_{k_l}}  ) \leq \# \left( F_{n_{k_l}} \, \triangle \, \left[F_{n_{k_l}} + (\w - \v)\right] \right) \, +
 \frac {1}{[\Z^d:G]}\,  \#(F_{n_{k_l}} ),
\]
which can be rewritten as
\[
 \frac{\beta - 1}{ [\Z^d : G]} \leq \frac{\# \left( F_{n_{k_l}} \, \triangle \, \left[F_{n_{k_l}} + (\w - \v)\right] \right)}{\# (F_{n_{k_l}})}.
\]
This too contradicts $\{F_n\}$ being F{\o}lner. 
\end{proof}


\subsection{Odometers}\label{odometers}  
We will especially consider a well-studied class of rank one systems called odometers; background on odometers where the acting group is $\mathbb{Z}$ can be found in \cite{Dow}.  $\Z^d$-odometers can be defined in a variety of ways; the approach we review here is a construction due to Cortez \cite{Cor}, but there is an equivalent characterization given by Giordano, Putnam and Skau \cite{GPS3}. 

For Cortez' construction, we begin by considering any decreasing sequence $\mathfrak{G} = \{G_j\}_{j = 1}^\infty$ of subgroups of $\Z^d$, where each $G_j$ has finite index in $\Z^d$.  For each $j \geq 1$, let $q_j: \Z^d/G_{j+1} \to \Z^d/G_j$ be the quotient map.  Then define
\begin{align*}
X_\gothG & = \stackrel[\longleftarrow]{}{\lim} (\Z^d/G_j) \\
& = \{(\x_1, \x_2, \x_3, ...) : \x_j \in \Z^d/G_j \textrm{ and } q_j(\x_{j+1}) = \x_j \textrm{ for all }j\}.
\end{align*}
$X_\gothG $ is a topological group, where the topology is the product of the discrete topologies on each $\Z^d/G_j$.   Also, for each $j \geq 1$ there is a natural coordinate map $\pi_j : X_\gothG \to \Z^d/G_j$, and we define $\mathcal{X}_\gothG$ on $X_\gothG$ to be the $\sigma$-algebra generated by the sets $\{\pi_j^{-1}(\x + G_j) : j \in \{1,2,3,...\}, \x \in \Z^d\}$.  More importantly, there is a minimal action $\sigma_\gothG : \Z^d \actson X_\gothG$ given by
\[
\sigma_\gothG^\v(\x_1, \x_2, \x_3, ...) = (\x_1 + \v + G_1, \x_2 + \v + G_2, \x_3 + \v + G_3, ...).
\]  

\begin{defn}
\label{CortezDef}
A \textbf{\emph{$\Z^d$-odometer}} is any $\Z^d$-action conjugate to one of the form $(X_\gothG, \sigma_\gothG)$ described above, where $\gothG$ is some decreasing sequence of finite-index subgroups of $\Z^d$.
\end{defn}

We remark that $G$-odometers can be defined for any residually finite group $G$ (not just $\Z^d$); see \cite{CP}. 

The following results can be found in Theorem 2.2 of \cite{GPS3}:

%
%

\begin{thm}[Basic properties of odometers] \label{basicodomprops} Let $\gothG = \{G_1, G_2, ...\}$ be a decreasing sequence of finite-index subgroups of $\Z^d$.
\begin{enumerate}
\item So long as $G_j \neq G_{j+1}$ for infinitely many $j$, $X_\gothG$ is a Cantor set, $(X_\gothG, \mathcal{X}_\gothG)$ is a standard Lebesgue space and $\sigma_\gothG$ is a minimal action of $\Z^d$ by homeomorphisms on $X_\gothG$;
\item $(X_\gothG, \sigma_\gothG)$ is free if and only if $\stackrel[j=1]{\infty}\bigcap G_j = \{\0\}$;
\item $(X_\gothG, \mathcal{X}_\gothG, \sigma_\gothG)$ is uniquely ergodic and its  invariant Borel probability measure $\mu_\gothG$ is Haar measure on $X_\gothG$.  This measure satisfies $\mu_\gothG \left(\pi_j^{-1}(\x + G_j)\right) = [\Z^d : G_j]^{-1}$ for all $j \geq 1$ and all $\x \in \Z^d$.
\end{enumerate}
\end{thm}

From Theorem \ref{basicodomprops}, we can classify $\Z^d$-odometers into two types.  First, if $X_\gothG$ is infinite (assured in the setting of (1) in Theorem \ref{basicodomprops}), we will decree $(X_\gothG, \sigma_\gothG)$ to be an \textbf{infinite $\Z^d$-odometer}.  On the other hand, if the sequence $\gothG$ is eventually constant, we obtain an odometer with only finitely many points in its phase space. This is equivalent to the following:

\begin{defn}
Suppose $G$ is a finite-index subgroup of $\Z^d$.  Let $\delta$ be uniform counting measure on the finite set $\Z^d/G$, and for each $\v \in \Z^d$, define $\tau^\v : \Z^d/G \to \Z^d/G$ to be $\tau^\v(\g + G) = \g + \v + G$.  We call any measure-preserving system conjugate to $\tau : \Z^d \actson (\Z^d/G, \delta)$ a \textbf{\emph{finite $\Z^d$-odometer}}.
\end{defn}

It is well known that odometers are rank one actions.  That said, we will make use of a specific rank one structure for a given $\Z^d$-odometer, which we describe in the following result.

%
%

\begin{thm} 
\label{Ggenerators}
Let $(X_\gothG, \sigma_\gothG)$ be the $Z^d$-odometer given by the decreasing sequence $\gothG = \{G_1, G_2, G_3, ...\}$ of finite-index subgroups of $\Z^d$.  Then:
\begin{enumerate}
\item There exist upper-triangular matrices $A_1, A_2, A_3, ... \in GL(d,\Z)$ so that each $G_j$ is generated by the columns of $A_j$; and
\item $(X_\gothG, \sigma_\gothG)$ is stacking rank one for a F{\o}lner sequence of rectangular shapes, where the error sets of the associated towers are empty.
\end{enumerate}
\end{thm}

\begin{proof} For each $j$, first let 
\[
a_{j,1,1} = \min \{ n>0: n \e_1 \in G_j\}.
\]
and set $a_{j,k,1} = 0$ for $k \in \{2, ..., d\}$.  (Here and throughout, $\e_j$ is the $j^{th}$ standard basis vector of $\mathbb{R}^d$.)  Since $G_j$ has finite index, the integer $a_{j,1,1}$ exists. 

Next, let
\[
a_{j,2,2} =  {\mbox{min}}\{ n>0: \exists\, i \textrm{ such that }n\e_2 + G_j = (i,0,...,0) + G_j \}.
\]
and set $a_{j,1,2}$ to be any value of $i$ that works in the previous line, meaning $a_{j,2,2}\e_2 + G_j = (a_{j,1,2},0,...,0) + G_j$.  Then, for any $k \in \{3, .., d\}$, set $a_{j,k,2} = 0$.

Continuing in this fashion, for each $l \in \{1,...,d\}$, define 
\begin{align*}
a_{j,l,l} & = {\mbox{min}}\{  n>0: \exists \, i_1, i_2, ..., i_{l-1} \textrm{ such that } \\
& \qquad n\e_l + G_j = (i_1,i_2,...,i_{l-1},0,...,0) + G_j \};
\end{align*}
then choose $a_{j,1,l}, a_{j,2,l}, ... a_{j, l-1, l}$ to be any values of $i_1, i_2, ..., i_{l-1}$ that work in the previous line, and finally, for any $k \in \{l+1, ..., d\}$, set $a_{j,k,l} = 0$.

Having done this, we obtain upper-triangular $d \times d$ matrices $A_j = \left(a_{j,k,l}\right)_{k,l}$ so that $G_j = A_j \Z^d$.  The columns of each of these matrices form a set of $d$ linearly independent vectors that generate $G_j$.

Furthermore, by setting $F_j = \{0, 1, ..., a_{j,1,1}-1\} \times \{0, 1, ..., a_{j,2,2} -1\} \times \cdots \times \{0, 1, ..., a_{j,d,d} - 1\}$, we see that the shape $F_j$ contains exactly one representative element from each coset in $\Z^d/G_j$.  So for each $j$, by setting $B_j = \pi_j^{-1}(\0 + G_j)$, the sets
$\{\pi_j^{-1}(\v + G_j) : \v \in F_j\}$ form a $\sigma_\gothG$-tower $\mathcal{T}_j$ of shape $F_j$ with base $B_j$, where the error set is empty.  The corresponding $\sigma_\gothG$-towers so obtained generate the $\sigma$-algebra $\mathcal{X}_\gothG$ by definition, meaning $(X_\gothG, \sigma_\gothG)$ is rank one.

Finally, to show the towers are stacking, let $\v \in F_{j+1}$, and then set $\widetilde{\v} \in F_j$ to be the unique representative in $F_j$ of the coset $q_j(\v + G_{j+1}) \in \Z^d/G_j$.  The level $\pi_{j+1}^{-1}(\v + G_{j+1})$ of $\mathcal{T}_{j+1}$ is a subset of the level $\pi_j^{-1}(\widetilde{\v} + G_j)$ of $\mathcal{T}_j$, so the towers stack as wanted.  \end{proof}


\subsection{Finite odometer factors}



An eventual goal of this work is to determine criteria that ensure when a F{\o}lner rank one $\Z^d$-system factors onto, or is conjugate to, a $\Z^d$-odometer.  Toward that end, it will be useful to consider when such a system factors onto a finite odometer.  For example, when the original system is itself an odometer, given by sequence $\gothG = \{G_1, G_2, G_3, ...\}$, it is easy to obtain a factor map from $(X_\gothG, \sigma_\gothG)$ onto each finite odometer $(\Z^d/G_j, \tau)$ simply by projecting $x\in X_\gothG$ onto its $j^{th}$ component.  In fact, we will see later that the set of $G$ for which an odometer factors onto $(\Z^d/G, \tau)$ determines the odometer up to conjugacy.  With this idea as motivation, we will examine the set of all finite odometer factors of a general action as follows, where $\mathcal{FF}$ refers to ``finite factors", a shortened version of ``finite odometer factors":

\begin{defn}\label{FiniteOdometerFactors}
Given the measure-preserving system $\T : \Z^d \actson (X,\mu)$, set
\[
\mathcal{FF}(\T) = \{G \leq \Z^d : [\Z^d : G] < \infty \textrm{ and } (X,\mu,\T) \textrm{ factors onto } (\Z^d/G, \delta, \tau)\}.
\]
\end{defn}
We begin with two general properties of the set $\FF(\T)$:

\begin{defn}
Let $\mathcal{F}$ be a set of subgroups of $\Z^d$.  We say $\mathcal{F}$ is \textbf{\emph{closed under supergroups}} if whenever $G \in \mathcal{F}$ and $G \leq H \leq \Z^d$, then $H \in \mathcal{F}$.  We say $\mathcal{F}$ is \textbf{\emph{closed under intersections}} if $G \in \F$ and $H \in \F$ imply $G \cap H$ must also be in $\F$.
\end{defn}

%
%

\begin{lemma} 
\label{propsofFF}
For any $\Z^d$-action $(X,\mu,\T)$, $\FF(\T)$ is closed under supergroups and closed under intersections.
\end{lemma}

\begin{proof} We begin by showing $\FF(\T)$ is closed under supergroups.  Let $G \in \FF(\T)$ and suppose $G \leq H$.  Thus there is a natural quotient map $q:(\Z^d/G, \delta_G, \tau_G) \to (\Z^d/H,  \delta_H, \tau_H)$, and therefore any factor map $\pi : (X,\mu,\T) \to (\Z^d/G,  \delta_G, \tau_G)$ yields a factor map $q \circ \pi: (X,\mu,\T)\to (\Z^d/H, \delta_H, \tau_H)$, meaning $H \in \FF(\T)$ as wanted.

We now show $\FF(\T)$ is closed under intersections.  Let $G,H \in \FF(\T)$, so that there are factor maps $\pi_G : (X,\mu,\T) \to (\Z^d/G,  \delta_G, \tau_G)$ and $\pi_H : (X,\mu,\T) \to (\Z^d/H,  \delta_H, \tau_H)$.  Now consider the partition $\P$ of $X$ coming from the values of $\pi_G$ and $\pi_H$.  Note that for each $\v \in \Z^d$, $\T^\v$ permutes the atoms of this partition.  

Now, given two atoms $E$ and $\widetilde{E}$ of $\P$, we say $E \sim \widetilde{E}$ if there is $\v \in \Z^d$ so that $\T^\v(E) = \widetilde{E}$; this gives an equivalence relation on $\P$ where we denote the equivalence class of atom $E$ by $[E]$.  For each equivalence class under the equivalence relation $\sim$, take an atom $E$ in that class; notice that $\bigcup_{\v \in \Z^d} \T^\v(E) = [E]$.  At the same time, $\T^\v(E) = E$ if and only if $\v \in G \cap H$, so each $x \in [E]$ can be written as $x = \T^\v(x')$ for $x' \in E$ where $\v$ is unique mod $(G \cap H)$.  This gives us a map $\pi_{G \cap H} : [E] \to \Z^d / (G \cap H)$ defined by
\[
\pi_{G \cap H}(x) = \v + (G \cap H) \iff x = \T^\v(x') \textrm{ for }x' \in E.
\]
As each $x \in X$ belongs to exactly one class $[E]$, we can extend $\pi_{G \cap H}$ 
 to be a map from $X$ to $\Z^d/(G \cap H)$.  The resulting map $\pi_{G \cap H}$ is  then a factor map from $(X,\mu,\T)$ to $(\Z^d/(G \cap H), \delta_{G \cap H}, \tau_{G \cap H}  )$, meaning $G \cap H \in \FF(\T)$ as desired.
\end{proof}


%
%

\section{Finite odometer factors of F{\o}lner rank one actions}

As was mentioned before, the Cortez definition of an odometer makes it is easy to see that there is a factor map $\pi_j$ from the odometer $X_\gothG$ onto the finite odometer $\Z^d/G_j$.   
Thus in order to explore the kinds of (general) odometers onto which a given $\Z^d$-action can factor, we will begin by studying the types of finite odometers onto which a  $\Z^d$-action can factor.  

%
%

\begin{thm} 
\label{finitefactorthm}
Let $\T : \Z^d \actson (X,\mu)$ be F{\o}lner rank one and let $G$ be a finite index subgroup of $\Z^d$.  Then the following are equivalent:
\begin{enumerate}
\item $G\in \mathcal{FF}(\T) $.
\item ``$\textrm{Eventually, most descendants are congruent mod } G$.''  By this, we mean that for every $\ep > 0$, there exists $N \in \N$ so that for all 
$n \geq m \geq N$, there is $\g  \in \Z^d$ such that
\[
\frac{\#(\{\i \in I_{m,n} : \i \not \equiv \g \mod G\})}{\#(I_{m,n})} < \ep.
\]
\end{enumerate}
\end{thm}

%
\begin{proof} (1) $\so$ (2): We begin with supposing that $\pi : X \to \Z^d/G$ is a factor map and recalling that since $X$ is F{\o}lner rank one, there exists a sequence of towers $\{ \mathcal{T}_n \}$ of shape $\{ F_n \}$ whose levels generate the $\sigma$-algebra.  Fix $\ep$ such that $0< \ep <1$; let $\tilde{\ep} = \frac{\ep}{2}$ and $\eta = \frac{\tilde{\ep}}2 [\Z^d:G]^{-1}$. 

We will show that (2) holds with the following three steps:
\begin{enumerate}
\item[\emph{(i)}] We begin by showing that there is an $N \in \N$ such that for every $m\ge N$, each level $L$ of $\mathcal{T}_m$ 
has associated to it some $\gamma + G$ such that 
\[
\frac{\mu\left( \, L  - \pi^{-1}(\gamma + G) \, \right)}{\mu(L )} < \tilde{\ep}.
\]
\item[\emph{(ii)}] We next show that having 
the coset $\gamma +G$ associated to level $B_m$ as described in step (i) means the same coset is associated to most $T^{\v}(B_n)$,
where $\v\in I_{m,n}$.
\item[\emph{(iii)}] Finally, we show that when two levels $T^{\v_1}(B_n)$ and 
$T^{\v_2}(B_n)$  of $\mathcal{T}_n$ are associated to the same coset, then $\v_1$ and $\v_2$ are congruent mod $G$.   
\end{enumerate}

{\it{Step (i):}}
Consider the sets $\{ \pi^{-1}(\gamma + G): \gamma + G\in \Z^d/G \}$.  These partition $X$ into sets of equal measure 
$[\Z^d: G]^{-1}$.  Since the levels of the towers generate the $\sigma$-algebra, there exists $N$ such that whenever $m\ge N$, each 
$\pi^{-1}(\gamma + G)$ is well-approximated by some union of levels in $\mathcal{T}_m$.  In other words, there is a union of levels, call it $\Lambda_{m,\gamma}$, such that
\begin{equation}\label{wellapprox}
\mu (\Lambda_{m,\gamma}\,\, \triangle \,\, \pi^{-1}(\gamma + G)\, ) < \eta.
\end{equation}
If we fix $\gamma$ for a moment and let $L$ be one of the levels in $\Lambda_{m,\gamma}$, we can use the fact that every other level 
$\widetilde{L} \notin \Lambda_{m, \gamma}$ equals 
$\T^{\v}(L)$ for some $\v$ to argue that $\widetilde{L} \in \Lambda_{m, \gamma + \v}$.  In other words, we can assume that 
every level in $\mathcal{T}_m$ lies in $\Lambda_{m,\tilde{\gamma}}$ for some $\tilde{\gamma}$.  We will use this to justify the following claim:

\noindent For every level $L\in \mathcal{T}_m$, there exists $\gamma + G \in \Z^d/G$ such that 
\[  
\frac{\mu\left( L \,\, - \,\, \pi^{-1}(\gamma + G)  \right)}{  \mu(L)} <  \tilde{\ep}. 
\]
We prove this claim by contradiction: suppose  there exists a level $L$ such that for all $\gamma + G$, 
\[  
\frac{\mu\left( L \,\, - \,\, \pi^{-1}(\gamma + G)  \right)}{  \mu(L)} \ge  \tilde{\ep}. 
\]
But then for all $\v \in \Z^d$, we also have 
\[  
\frac{\mu\left( \T^{\v}(L) \,\, - \,\, \pi^{-1}(\v + \gamma + G)  \right)}{  \mu(L)} \ge  \tilde{\ep}. 
\]
which we can rewrite as saying that  for every every $\v \in \Z^d$ and $\tilde{\gamma} + G\in \Z^d/G$, we have 
\[
\frac{\mu\left( \T^{\v}(L) - \pi^{-1}(\tilde{ \gamma} + G)  \right)}{  \mu(\T^\v(L))} \ge \tilde{\ep}. 
\]
Let $\gamma$ be such that $L \subseteq \Lambda_{m,\gamma}$.  Using the above, we can say that for this particular $\gamma$ and for every $\widetilde{L} \in \Lambda_{m,\gamma}$,
\begin{equation} \label{wellapprox2}
\mu (\widetilde{L} - \pi^{-1}({ \gamma} + G)) \geq \tilde{\ep} \, \mu(\widetilde{L}).
\end{equation} 
We thus have the following chain of inequalities, which ultimately yield the contradiction $\eta > \eta$:
\begin{align*}
\eta & > \mu (\Lambda_{m,\gamma} \,  \triangle \, \pi^{-1}(\gamma + G)\, ) \qquad \textrm{(from (\ref{wellapprox}) above)}\\
& \geq \mu (\Lambda_{m,\gamma} -  \pi^{-1}(\gamma + G)\, ) \\
& = 
\sum_{\widetilde{L} \in \Lambda_{m,\gamma}} \mu (\widetilde{L} -  \pi^{-1}({ \gamma} + G)) \\
& \geq   \sum_{\widetilde{L} \in \Lambda_{m,\gamma}} \tilde{\ep} \, \mu(\widetilde{L}) \qquad \textrm{(from (\ref{wellapprox2}) above)}\\
& = \tilde{\ep} \, \mu ( \Lambda_{m,\gamma} ) \\
& \geq \tilde{\ep} \, \mu \left( \Lambda_{m,\gamma} \cap \pi^{-1}(\gamma + G)\right) \\
& = \tilde{\ep} \, \left[\mu(\pi^{-1}(\gamma + G)) - \mu(\pi^{-1}\left(\gamma + G) - \Lambda_{m,\gamma}\right) \right] \\
& = \tilde{\ep} \, \left[[\Z^d : G]^{-1} - \mu(\pi^{-1}(\gamma + G) - \Lambda_{m,\gamma}) \right] \\ 
& \geq \tilde{\ep} \, \left[[\Z^d : G]^{-1} - \mu(\pi^{-1}(\gamma + G) \, \triangle \, \Lambda_{m,\gamma}) \right] \\ 
& \geq \tilde{\ep} \, \left[[\Z^d : G]^{-1} - \eta \right] \qquad \textrm{(again using (\ref{wellapprox}))} \\ 
& = \tilde{\ep} \, \left[[\Z^d : G]^{-1} - \frac 12 \tilde{\ep} \, [\Z^d : G]^{-1}\right] \qquad \textrm{(by the definition of $\eta$)} \\
& > \tilde{\ep} \, \left[[\Z^d : G]^{-1} - \frac 12  [\Z^d : G]^{-1}\right] \qquad \textrm{(since $\tilde{\ep} < \ep < 1$)} \\
& = \frac 12 \tilde{\ep} \, [\Z^d : G]^{-1} \\
& = \eta.
\end{align*}
This proves the claim, and finishes the first step of the proof. \\

\vspace{.1in}


{\it{Step (ii):}} Note that the first step says that for  the base level $B_m$ of $\mathcal{T}_m$, where $m\ge N$, there is some coset, which we will denote by $\widetilde{\g}+G$, such that 
\begin{equation} \label{contained}
\frac{\mu\left( B_m \,\, - \,\, \pi^{-1}(\widetilde{\g}+G)  \right)}{  \mu(B_m)} <  \tilde{\ep}. 
\end{equation}
We know that, for any $n\ge m$, $B_m$ is divided into subsets that become various levels in $\mathcal{T}_n$.  Those levels are the sets $\T^{\v}(B_n)$, where $\v\in I_{m,n}$.
We want to say that since (\ref{contained}) holds for $B_m$, then the same inequality with $B_m$ replaced by $\T^{\v}(B_n)$ 
holds for most $\v\in I_{m,n}$.  To this end, define
$$I_{\textrm{bad}} = \{\v \in I_{m,n} : \frac{\mu \left( \T^{\v}(B_n) -   \pi^{-1}(\widetilde{\g}+G)   \right) }{\mu \left( \T^{\v}(B_n) \right) } \ge \tilde{\ep} \}.$$
We will complete this step in the proof by showing that 

\[  \frac{\#(I_{\textrm{bad}})}{\#(I_{m,n})} < 2\tilde{\ep}. \]  
To do this, first note that $\stackbin[\v\in I_{\textrm{bad}}]{}{\bigcup} \T^{\v} (B_n) \subseteq B_m$ and thus 
\[  
\bigcup_{\v\in I_{\textrm{bad}}}  \left( \T^{\v}(B_n) -   \pi^{-1}(\widetilde{\g}+G)  \right) \, \subseteq \,  B_m -  \pi^{-1}(\widetilde{\g}+G)
\]
which implies
\[
\mu \left( \, \bigcup_{\v\in I_{\textrm{bad}}}  \left( \T^{\v}(B_n) -   \pi^{-1}(\widetilde{\g}+G)  \right) \, \right) \,  \le \,  \mu \left( B_m -  \pi^{-1}(\widetilde{\g}+G)  \, \right).
\]
But although $\v \in I_{\textrm{bad}}$ means that 
$\mu \left( \T^{\v}(B_n) -   \pi^{-1}(\widetilde{\g}+G)   \right) /\mu \left( \T^{\v}(B_n) \right)   \ge \tilde{\ep}$, Step (i) assured us there is some 
$\gamma+G \in \Z^d/G$ with 
\[
\frac{\mu \left( \T^{\v}(B_n) -   \pi^{-1}(\gamma+G)   \right)}{\mu \left( \T^{\v}(B_n) \right)}   < \tilde{\ep}. \]
 For this $\gamma + G$, divide $\T^{\v}(B_n)$ into two disjoint subsets: the part in $\pi^{-1}(\gamma+G)$ and the part not in $\pi^{-1}(\gamma+G)$.  Since
\[
\mu \left( \T^{\v}(B_n) - \pi^{-1}(\gamma+G)  \right) < \tilde{\ep} \, \mu\left( \T^{\v}(B_n)  \right), 
\]
we have $\mu \left( \T^{\v}(B_n) \cap \pi^{-1}(\gamma+G)  \right) > (1-\tilde{\ep}) \, \mu\left( \T^{\v}(B_n)  \right)$ and we can then say
\begin{align*}
\mu \left( \T^{\v}(B_n) - \pi^{-1}(\widetilde{\g}+G)  \right)  
& \ge \mu \left( \T^{\v}(B_n) \cap \pi^{-1}(\gamma+G)  \right) \\
& > (1-\tilde{\ep}) \, \mu\left( \T^{\v}(B_n)  \right)\\
& =  (1-\tilde{\ep}) \, \mu(B_n).
\end{align*}
Therefore
\begin{align*}
\mu \left( \bigcup_{\v\in I_{\textrm{bad}}}  \left( \T^{\v}(B_n) -   \pi^{-1}(\widetilde{\g}+G) \right)  \right) & = 
\sum_{\v\in I_{\textrm{bad}}} \mu \left(  \T^{\v}(B_n) -   \pi^{-1}(\widetilde{\g}+G)  \right) \\
& > \sum_{\v\in I_{\textrm{bad}}} (1-\tilde{\ep}) \, \mu(B_n) \\
& = \#(I_{\textrm{bad}}) \cdot (1-\tilde{\ep}) \, \mu(B_n) \\
& =    (1-\tilde{\ep}) \cdot \frac{ \#(I_{\textrm{bad}}) }{\#(I_{m,n})} \,\cdot \, \#(I_{m,n}) \,        \mu(B_n) \\
& > \frac{\#(I_{\textrm{bad}})}{\#(I_{m,n})}  (1-\tilde{\ep}) \mu(B_m) .
 \end{align*}
 Now suppose the claim is false, meaning $\frac{\#(I_{\textrm{bad}})}{\#(I_{m,n})} \ge 2\tilde{\ep}$.  This would mean 
\[
\mu \left(  B_m - \pi^{-1}(\widetilde{\g}+G) \right) \ge \mu \left( \bigcup_{\v\in I_{\textrm{bad}}}  \left( \T^{\v}(B_n) -   \pi^{-1}(\widetilde{\g}+G) \right)  \right)
> \, 2\tilde{\ep}\, (1-\tilde{\ep}) \, \mu(B_m).
\]
But $\tilde{\ep} < \frac{1}{2}$, so $1-\tilde{\ep} > \frac{1}{2}$, and we thus have 
\[
\mu \left(  B_m - \pi^{-1}(\widetilde{\g}+G) \right) > \tilde{\ep} \, \mu(B_m),
\]
contradicting (\ref{contained}).

\vspace{.1in}


{\it{Step (iii):}}  Let $\v \in I_{m,n} - I_{\textrm{bad}}$ and let $\g \in \Z^d$ be such that $\v \equiv \g \mod G$.  Take any other
$\i \in I_{m,n} - I_{\textrm{bad}}$.  Since we have both 
\[
\frac{\mu \left( \T^{\v}(B_n) -   \pi^{-1}(\widetilde{\g}+G)   \right) }{\mu \left( \T^{\v}(B_n) \right) } < \tilde{\ep} \,\, {\textrm{ and }} \frac{\mu \left( \T^{\i}(B_n) -   \pi^{-1}(\widetilde{\g}+G)   \right) }{\mu \left( \T^{\i}(B_n) \right) } < \tilde{\ep},
\]
we can find $x \in  \T^{\v}(B_n) \cap \pi^{-1}(\tilde{\g}+G)$ such that $\T^{\i-\v}(x) \in \T^{\i}(B_n) \cap \pi^{-1}(\tilde{\g}+G)$.
Thus $\pi(x) \in \tilde{\g}+G$ and $\pi (\T^{\i-\v}x ) \in \tilde{\g}+G$.  But
\[
\pi (\T^{\i-\v}x ) = \i - \v + \pi(x) = \i - \v + \tilde{\g}+G.
\]
This implies that $\i - \v \in G$, or $\i \equiv \v$ mod $G$.  So $\i \equiv \g$ mod $G$ as well.  In other words, all $\i \in I_{m,n} - I_{\textrm{bad}}$ are congruent to $\g$ mod $G$ and so 
\[
\frac{\#(\{\i \in I_{m,n} : \i \not \equiv \g \mod G\})}{\#(I_{m,n})} \leq \frac{\#(I_{\textrm{bad}})}{\#(I_{m,n})} < 2 \tilde{\ep} < \ep,\]
as wanted.

\vspace{.2in}

%
(2) $\so$ (1): We are now assuming that for every $\ep$, there exists $N \in \N$ such that for all $n \geq m \geq N$, there is $\g  \in \Z^d$ such that
\begin{equation}  \label{equation21}
\frac{\#(\{\i \in I_{m,n} : \i \not \equiv \g \mod G\})}{\#(I_{m,n})} < \ep.
\end{equation}
We want to show that $G\in \mathcal{FF}(\T) $, i.e. that $(X,\mu,\T)$ factors onto  $(\Z^d/G,  \delta, \tau)$.

For each $i \in \{1,2,3,...\}$, let $\ep_i = \frac{1}{2^{i+1}}$ and apply the assumption to obtain 
 the corresponding $N_i$ and $\g_i$.
Without loss of generality the sequence $\{N_i\}$ can be taken to be increasing, and can be chosen large enough so that the measure of  stage $N_i$-tower
$\mathcal{T}_{N_i}$ is at least $1 - \frac{1}{2^{i+1}}$.  
Also, since $\{F_n\}$ is a F{\o}lner sequence, from Lemma \ref{folnerlemma}, we can assume that for all $i$, $F_{N_i}$  intersects every coset in $\Z^d/G$.

In order to find the wanted factor map from $(X,\mu,\T)$ onto $(\Z^d/G,  \delta, \tau)$, we begin by first defining what might be the most obvious map from $\mathcal{T}_{N_i}$ onto $\Z^d/G$.  As we are assuming the descendents of the base $B_{N_i}$ are mostly equal mod $\g_i+G$, we then consider the set of points where the map on $\mathcal{T}_{N_i+1}$ is not simply the map on 
$\mathcal{T}_{N_i}$ adjusted by this $\g_i+G$, and show this set has small measure.  This will then let us define a factor map on all of $X$ by using the sum of these adjustments.  Now for the details.

We begin by defining, for each $i$, a function $\phi_i : \mathcal{T}_{N_i} \to \Z^d/G$ by setting $\phi_i(x) = \v + G$ whenever $x \in \T^\v(B_{N_i})$.  Observe that since $F_{N_i}$ intersects every coset in $\Z^d/G$, each $\phi_i$ is onto.  Now set
\[
D_i = \{ x\in \mathcal{T}_{N_i} : \phi_{i+1}(x) \neq \phi_i(x) + (\g_i + G)\}.
\]
We want to show that this has small measure. \\

\emph{Claim:}  $\mu(D_i) < \ep_i$. \\

\emph{Proof of claim:} Let $F_{N_i}$ be the shape corresponding to $\mathcal{T}_{N_i}$.  We know most $\w\in I_{N_i,N_{i+1}}$ are such that $\w \equiv \g_i$ mod $G$.  For such $\w$, consider the points at these levels and call their union $B_{\textrm{good}}$.  That is, set
\[
B_{\textrm{good}} = \{x \in \T^\w(B_{N_{i+1}}) : \w\in I_{N_i,N_{i+1}} \textrm{ and } \w \equiv \g_i \mod G\}.
\]
Note $B_{N_i} = B_{\textrm{good}} \, \sqcup \, ( B_{N_i}  - B_{\textrm{good}})$, and observe that for any 
$x \in \stackrel[\v \in F_{N_i}]{}{\bigsqcup} \T^\v(B_{\textrm{good}})$,  $x \in \T^{\v}(\T^{\w}(B_{N_{i+1}})) =  \T^{\v + \w}(B_{N_{i+1}})$.
Thus $\phi_{i+1}(x) = \v+\w+G = \v+\g_i+G = \phi_i(x) + \g_i+G$, and so $x \notin D_i$.  In other words, the set $D_i$ does not include any points from 
$\stackrel[\v \in F_{N_i}]{}{\bigsqcup} \T^\v(B_{\textrm{good}})$.  We can thus say,
\begin{align*}
\mu(D_i)  & \leq \mu \left(\bigsqcup_{\v \in F_{N_i}} \T^\v(B_{N_i} - B_{\textrm{good}}) \right) \\
& = \sum_{\v \in F_{N_i}} \mu\left( \T^\v(B_{N_i} - B_{\textrm{good}}) \right) \\
& = \sum_{\v \in F_{N_i}} \mu \left( B_{N_i} - B_{\textrm{good}}  \right).
\end{align*}
Recall that $B_{N_i}$ is divided into subsets and placed at levels $I_{N_i, N_{i+1}}$ in $\mathcal{T}_{N_{i+1}}$.  
Some of these levels are at heights congruent to $\g_i$ mod $G$, and those would not be included in the set $B_{N_i} - B_{\textrm{good}}$.  
The proportion of the remaining levels in $I_{N_i, N_{i+1}}$ is bounded by (\ref{equation21}), and thus we can say
\[
\mu \left( B_{N_i} - B_{\textrm{good}}  \right)  \le \mu(B_{N_{i+1}}) \,  \ep_i \, \#(I_{N_i, N_{i+1}}) = \mu(B_{N_i}) \,\ep_i .
\]
Putting this together with the above we see that
\[ \mu(D_i)  \leq \#(F_{N_i}) \mu(B_{N_i}) \,\ep_i \le \mu(X)  \,\ep_i = \ep_i, \] completing the proof of the claim.

\vspace{.1in}

We next adjust the functions $\phi_i$ by the amounts $\g_1, \g_2, ...\g_{i-1}$.  More specifically, set $\gamma_i \in \Z^d/G$ to be 
$\gamma_i = \stackrel[j=1]{i}{\sum} (\g_j + G)$.  
Then define, for each $i$, a function $\pi_i : \mathcal{T}_{N_i} \to \Z^2/G$ by setting
\[
\pi_i(x) = \phi_i(x) - \gamma_{i-1}.
\]
Since each $\phi_i$ is surjective, so is each $\pi_i$.  Now observe
\begin{align*}
x \in D_i & \iff \phi_{i+1}(x) \neq \phi_i(x) + (\g_{i} + G) \\
& \iff \phi_{i+1}(x) - (\g_{i} + G) \neq \phi_i(x)  \\
& \iff \phi_{i+1}(x) + \gamma_{i-1} - \gamma_i \neq \phi_i(x)  \\
& \iff \phi_{i+1}(x)  - \gamma_i \neq \phi_i(x) - \gamma_{i-1}\\
& \iff \pi_{i+1}(x) \neq \pi_i(x).
\end{align*}
So 
\begin{equation}
\label{newdefn}
\mu\left(\{x \in \mathcal{T}_{N_i} : \pi_{i+1}(x) \neq \pi_i(x)\}\right) = \mu(D_i) <  \ep_i = \frac 1{2^{i+1}}.
\end{equation} 

Finally, we can define the factor map from $X$ to $\Z^d/G$.  Note that the Borel-Cantelli Lemma says that almost every $x \in X$ belongs to only finitely many of the sets 
$\{x \in \mathcal{T}_{N_i} : \pi_{i+1}(x) \neq \pi_i(x)\}$.  In other words, for almost every $x \in X$ the sequence $\{\pi_i(x)\}$ eventually stabilizes.  We can thus define $\pi : X \to \Z^d/G$ by $\pi(x) = \stackrel[i \to \infty]{}{\lim} \pi_i(x)$.  

It only remains to show that this $\pi$ is indeed a factor map.  To do this, fix $x \in X$ and $\v \in \Z^d$.   Choose $i$ large enough so that $x$ and $\T^\v(x)$ both belong to $\mathcal{T}_{N_i}$, with $x$ at level $\w$ and $\T^\v(x)$ at level $\w + \v$, and that $\pi_i(x) = \pi(x)$ and $\pi_i(\T^\v(x)) = \pi(\T^\v(x))$.  Then
\begin{align*}
\pi(\T^\v(x)) & = \pi_i(\T^\v(x)) \\
& = \phi_i(\T^\v(x)) - \gamma_{i-1} \\
& = (\v + \w + G) - \gamma_{i-1} \\
& = [ \, (\w + G) - \gamma_{i-1}\, ] + (\v + G) \\
& = [\, \phi_i(x) - \gamma_{i-1}\, ] + (\v + G) \\
& = \pi_i(x) + (\v + G) \\
& = \pi(x) + (\v + G) \\
& = \tau^\v(\pi(x))
\end{align*}
so $\pi$ intertwines $\T$ and $\tau$ as desired. 
\end{proof}

We next turn to an example which shows that if one replaces ``F{\o}lner rank one'' with ``stacking rank one'' (or just ``rank one''), the $(2) \so (1)$ direction of Theorem \ref{finitefactorthm} fails. 

%
%

\begin{thm}
\label{stackingnonex}
There exists a $\T : \Z^2 \actson (X,\mu)$ which is stacking rank one and a finite-index subgroup $G \leq \Z^2$ so that for every $\ep > 0$, there exists $N \in \N$ such that for all $n \geq m \geq N$, there is $\g  \in \Z^2$ such that
\begin{equation}
\label{eqnstack}
\frac{\#(\{\i \in I_{m,n} : \i \not \equiv \g \mod G\})}{\#(I_{m,n})} < \ep,
\end{equation}
but $(X,\mathcal{X},\mu,\T)$ does not factor onto $(\Z^2/G, \delta, \tau)$.
\end{thm}
\begin{proof} Our system $(X, \T)$ will be the $\Z^2$-odometer $(X_\gothG, \sigma_\gothG)$ associated to the sequence of subgroups
\[\gothG = \{ 2\Z \times \Z, 4\Z \times \Z, ..., 2^n\Z \times \Z, ...\}.\]
In particular,  $\T^{(0,1)} = \sigma_\gothG^{(0,1)}$ is the identity map for this odometer.  Set $G = 2\Z \times 2\Z$. \\

Given any $\ep > 0$, choose $N = 0$, and for any $n \geq m \geq 0$, let $\g = \0$.  Note that for any $n \geq m \geq 0$, 
$I_{m,n} = \{(2^mc,0) : 0 \leq c < 2^{n-m}\}$ so
\[\#\left(\i \in I_{m,n} : i \not \equiv \0 \mod G\right) = \#(\emptyset) = 0,\]
ensuring (\ref{eqnstack}). \\

Finally, suppose $\pi : X_\gothG \to \Z^2/G$.  Then, for any $x \in X_\gothG$,
\begin{align*}
\pi(\T^{(0,1)}(x)) = \pi(x) \neq \pi(x) + (0,1) = \tau^{(0,1)}(\pi(x))
\end{align*}
so $\pi$ cannot intertwine $\T$ and $\tau$.  Thus no factor map from $(X,\T)$ to $(\Z^2/G, \tau)$ exists.

\end{proof}

%
%
In \cite{FGHSW} the authors note that for rank one $\Z$-actions, the action has a nontrivial finite factor if and only if the action is not totally ergodic.  To put this in the terminology of Theorem \ref{finitefactorthm}, we would say that a $\Z$-action is not totally ergodic if and only if there is a nontrivial finite index $G < \Z$ for which $G\in \mathcal{FF}(\T)$.  Since this is essentially a consequence of the fact that $\Z$ has no nontrivial subgroups of infinite index, it is not a surprise to find that a similar characterization does not hold for $\Z^d$-actions. We make this clear with the next example.

\begin{thm} \label{examplenottotallyergodic} There is a F{\o}lner rank one action $\T : \Z^2 \actson (X,\mu)$ which is not totally ergodic, but does not factor onto any nontrivial finite odometer.
\end{thm}
\begin{proof} To construct $\T$, start with Chacon's \cite{Cha2} map $T: \Z \actson (Y,\nu)$  and let $\T$ be the product of $T$ with itself, i.e. $\T : \Z^2 \actson (Y \times Y, \nu \times \nu)$ is $\T^\v(y,y) = \T^{(v_1, v_2)}(y,y) = (T^{v_1}(y), T^{v_2}(y))$.  

$(X,\T)$ is F{\o}lner rank one, with square tower shape $\{0, ..., h_m-1\}^2$ where $h_1 =1$ and  $h_{m+1} = 3h_m+1$ for all $m \geq 1$.  It is clearly not totally ergodic; indeed, for any measurable set $A \subseteq Y$, $A \times Y$ is $\T^{\e_2}$-invariant.

Now suppose that $G \in \FF(\T)$.  We will derive a contradiction by applying Theorem \ref{finitefactorthm}.  To do this, let $\ep < \frac 13$ and observe that for each $m$ and $n$ we can write $I_{m,n} = S_{m,n} \times S_{m,n}$ where $S_{m,m+1} = \{0, h_m, 2h_m+1\}$ and for all $n \geq m+1$,
\[
S_{m,n+1} = S_{m,n} \sqcup \left(S_{m,n} + h_{n}\right) \sqcup \left(S_{m,n} + 2h_{n}+1\right).
\]
Note that $\frac 23$ of the $3^{n}$ elements of $S_{m,m+n}$ come in pairs separated by $h_m$, so $\frac 23$ of the $9^n$ elements of $I_{m,n}$ come in pairs separated by $(h_m,0)$.  Therefore $(h_m,0) \equiv \0 \mod G$, as otherwise, condition (2) of Theorem \ref{finitefactorthm} would be violated.

Consequently,  $(h_m,0) \in G$ for all $m$, and therefore $(h_{m+1},0) = (3h_m+1,0) \in G$ as well.  Therefore $(1,0) = (h_{m+1},0) - 3(h_m,0) \in G$ as well.  A similar argument shows $(0,1) \in G$, so $G = \Z^2$.  This means $\T$ has no (nontrivial) finite factors.
\end{proof}

%
%

\section{Rank one $\Z^d$ systems that factor onto an odometer}

The last section gives us a characterization of when a Folner rank one $\Z^d$-action factors onto a finite odometer.  We now make use of this to describe when such a system factors onto an arbitrary odometer.

%
%

\begin{thm} 
\label{factorexists}
Let $\gothG = \{G_1, G_2, G_3, ...\}$ be a decreasing sequence of finite-index subgroups of $\Z^d$.
Let $\T : \Z^d \actson (X,\mu)$.  $(X,\mu,\T)$ factors onto the odometer $(X_\gothG, \mu_\gothG, \sigma_\gothG)$ if and only if for every $k$, $(X,\mu,\T)$ factors onto the finite action $(\Z^d/G_k, \delta_k, \tau_k)$.
\end{thm}
\begin{proof} $(\so)$ Let $\pi : (X,\T) \to (X_\gothG, \sigma_\gothG)$ be a factor map.  Then by composing $\pi$ with each coordinate map $\pi_k : X_\gothG \to \Z^d/G_k$, we obtain factor maps $\pi_k \circ \pi : (X,\T) \to (\Z^d/G_k,\tau_k)$ as wanted. \\

$(\os)$ Suppose that for each $k$, there is a factor map $\pi_k : (X,\T) \to (\Z^d/G_k,  \tau_k)$.
Because $X_\gothG$ is the inverse limit of the $\Z^d/G_k$, if the $\pi_k$ were commensurate with the quotient maps $q_k : \Z^d/G_{k+1} \to \Z^d/G_k$ described in Section 2.3 (so $q_k \circ \pi_{k+1} = \pi_k$ for all $k \geq 1$), then we could easily use these to define a factor map from 
$X$ to $X_\gothG$.  But our given factor maps have no particular relationship to one another and thus we will need to construct other factor maps 
$\psi_k : X \to \Z^d/G_k$ which are commensurate with the quotient maps.  We will do this by first defining a function which, in a sense, keeps track of how far off the given factor maps are from being commensurate with the quotient maps, and then shifting the given factor maps by this amount.  The details are as follows.

Begin by setting $\psi_1 = \pi_1$ and assume that factor maps $\psi_1, ..., \psi_n$ have been so constructed, so that $q_k \circ \psi_{k+1} = \psi_k$ for all $k < n$. Define $\r_{n+1}: X \to \Z^d / G_n$ by
\[    \r_{n+1}(x) = q_n \circ \pi_{n+1}(x) - \psi_n(x).   \]
Note that $\r_{n+1}$ is $\T$-invariant, because for any $\v \in \Z^d$, we have
\begin{align*}
\r_{n+1}(\T^\v(x)) & = q_n \circ \pi_{n+1}(\T^\v(x)) - \psi_n(\T^\v(x)) \\
& = q_n \circ \tau^\v(\pi_{n+1}(x)) - \tau^{\v}(\psi_n(x)) \\
& = q_n \left(\pi_{n+1}(x) + (\v + G_{n+1})\right) - (\psi_n(x) + (\v + G_n)) \\
& = q_n(\pi_{n+1}(x)) + q_n(\v + G_{n+1}) - \psi_n(x) - (\v + G_n) \\
& = q_n(\pi_{n+1}(x)) + (\v + G_n) - \psi_n(x) - (\v + G_n) \\
& = q_n \circ \pi_{n+1}(x) - \psi_n(x) \\
& = \r_{n+1}(x).
\end{align*}
We remark that if $\T$ is ergodic, a situation including all cases where $\T$ has rank one, then $\r_{n+1}$ is constant.  

Next define a map $c_{n+1}:\Z^d/G_n \to \Z^d/G_{n+1}$ which maps cosets to cosets in a way that is consistent with the quotient map.  More specifically, for each $\mathbf{s} \in \Z^d/G_n$, set 
$c_{n+1}(\mathbf{s}) \in \Z^d/G_{n+1}$  to be such that $q_n(c_{n+1}(\mathbf{s})) = \mathbf{s}$.  Use this to define 
$\widehat{\r}_{n+1} : X \to \Z^d/G_{n+1}$ as  
\[ \widehat{\r}_{n+1}(x) = c_{n+1}(\r_{n+1}(x));  \] 
the function $\widehat{\r}_{n+1}$ so defined is measurable and $\T$-invariant (and is therefore constant whenever $\T$ is ergodic) and satisfies $q_n \circ \widehat{\r}_{n+1} = \r_{n+1}$. 

Finally, define $\psi_{n+1} : X \to \Z^d/G_{n+1}$ by
\[\psi_{n+1}(x) = \pi_{n+1}(x) - \widehat{\r}_{n+1}(x).\]
For any $\v \in \Z^d$, we see
\begin{align*}
\psi_{n+1}(\T^\v(x)) & = \pi_{n+1}(\T^\v(x)) - \widehat{\r}_{n+1}(\T^\v(x)) \\
& = \tau_{n+1}^\v(\pi_{n+1}(x)) - \widehat{\r}_{n+1}(x) \\
& = \pi_{n+1}(x) + (\v + G_{n+1}) - \widehat{\r}_{n+1}(x) \\
& = \tau_{n+1}^\v(\, \pi_{n+1}(x) - \widehat{\r}_{n+1}(x)\, ) \\
& = \tau_{n+1}^\v(\psi_{n+1}(x)),
\end{align*}
so $\psi_{n+1}$ intertwines $\T$ and $\tau_{n+1}$.  Next,
\begin{align*}
q_n \circ \psi_{n+1}(x) & = q_n(\pi_{n+1}(x) - \widehat{\r}_{n+1}(x)) \\
& = q_n(\pi_{n+1}(x)) - q_n(\widehat{\r}_{n+1}(x)) \\
& = q_n(\pi_{n+1}(x)) - \r_{n+1}(x) \\
& = q_n \circ \pi_{n+1}(x) - \left[q_n \circ \pi_{n+1}(x) - \psi_n(x)\right] \\
& = \psi_n(x),
\end{align*}
so the maps $\psi_n$ are commensurate with one another.  We can then define a factor map $\psi : X \to X_\gothG$ by $\psi(x) = (\psi_1(x), \psi_2(x), \psi_3(x), ...)$, as wanted.
\end{proof}

%
%

\begin{cor}
\label{rankoneodometerfactor}
Let $\T : \Z^d \actson (X,\mu)$ be F{\o}lner rank one and let $ \sigma_\gothG : \Z^d \actson (X_\gothG, \mu_\gothG)$ be the $\Z^d$-odometer corresponding to the sequence of subgroups $\gothG = \{G_1, G_2, G_3, ...\}$.  Then the following are equivalent:
\begin{enumerate}
\item $(X,\mu,\T)$ factors onto $(X_\gothG, \mu_\gothG,\sigma_\gothG)$.
\item ``For each $j$, eventually most descendants are congruent mod $G_j$.''  This means that for each $j$, and for every $\ep > 0$, there exists $N \in \N$ such that for all $n \geq m \geq N$, there is $\g \in \Z^d$ such that
\[\frac{\#(\{\i \in I_{m,n} : \i \not \equiv \g \mod G_j\})}{\#(I_{m,n})} < \ep.\]
\end{enumerate}
\end{cor}
\begin{proof} This is immediate from Theorem \ref{finitefactorthm} and Theorem \ref{factorexists}. \end{proof}

\vspace{.1in}

The above lets us say, given a specific $\Z^d$-odometer, whether or not $(X,\mu,\T)$ factors onto it.  We next want a condition that says if there is {\it{some}} $\Z^d$-odometer onto 
which  $(X,\mu,\T)$ will factor.
Since factoring onto a finite odometer is basically taken care of by Theorem \ref{finitefactorthm}, we focus our attention here 
to infinite odometers.  

\vspace{.1in}

We first give a complete description of $\mathcal{FF}(X_\gothG)$.

%
%

\begin{thm}
\label{FFofodoms}
Let $(X_\gothG, \sigma_\gothG)$ be the $\Z^d$-odometer given by $\gothG = \{G_1, G_2, G_3, ...\}$.  Then
\[
\FF(\sigma_\gothG) = \{ H \leq \Z^d : [\Z^d : H] < \infty \textrm{ and } H \geq G_N \textrm{ for some }N\}.\]
\end{thm}
\begin{proof}  
$(\subseteq)$ First, recall from Theorem \ref{Ggenerators} that $\sigma_\gothG$ is rank one for a sequence of rectangular towers $\{\mathcal{T}_j\}$, where each level of tower $\mathcal{T}_j$ corresponds to a coset $\mod G_j$.  

Now, let $\tilde{H} \in \FF(\sigma_\gothG)$, meaning there is a factor map $\pi : (X_\gothG, \sigma_\gothG) \to (\Z^d/\tilde{H}, \tau)$.  This produces a measurable partition $\mathcal{P} = \{\pi^{-1}(\v + \tilde{H}) : \v + \tilde{H} \in \Z^d/\tilde{H}\}$ of $X_\gothG$, and since $\sigma_\gothG$ is rank one, we can therefore find an $N$ so that for each $\v + \tilde{H} \in \Z^d/\tilde{H}$, there are sets $E_\v$, each a union of levels of tower $\mathcal{T}_N$, so that
\[
\mu_\gothG(E_\v \, \triangle \, \pi^{-1}(\v + \tilde{H})) < \frac 14[\Z^d:\tilde{H}]^{-1} = \frac 14 \mu_\gothG(\pi^{-1}(\v + \tilde{H})).
\]
Now let $\g \in G_N$.  This means that for each $\v + \tilde{H} \in \Z^d/\tilde{H}$, $\sigma_\gothG^\g(E_\v) = E_\v$.
Also,  $\pi^{-1}(\g + \tilde{H}) = \pi^{-1}(\tau^\g(\0 + \tilde{H})) = \sigma_\gothG^\g(\pi^{-1}(\0 + \tilde{H}))$, 
so 
\begin{align*}
\mu_\gothG\left(\pi^{-1}(\g + \tilde{H}) \, \triangle \, \pi^{-1}(\0 + \tilde{H})\right)
& = \mu_\gothG\left(\sigma_\gothG^\g(\pi^{-1}(\0 + \tilde{H})) \, \triangle \, \pi^{-1}(\0 + \tilde{H})\right)\\
& \leq \mu_\gothG\left(\sigma_\gothG^\g(\pi^{-1}(\0 + \tilde{H})) \, \triangle \, \sigma_\gothG^\g(E_0)\right) \\
& \qquad + \mu_\gothG\left(\sigma_\gothG^\g(E_0) \, \triangle \, E_0\right) \\
& \qquad +  \mu_\gothG\left(E_0 \, \triangle \, \pi^{-1}(\0 + \tilde{H})\right) \\
& < \frac 14[\Z^d:\tilde{H}]^{-1} + 0 +  \frac 14[\Z^d:\tilde{H}]^{-1} \\
& = \frac 12 [\Z^d:\tilde{H}]^{-1} \\
& = \frac 12 \mu_\gothG\left(\pi^{-1}(\0 + \tilde{H})\right).
\end{align*}
That means that $\pi^{-1}(\g +\tilde{H})$ and $\pi^{-1}(\0 + \tilde{H})$ intersect in a set of positive measure, which implies that $\g + \tilde{H} = \0 + \tilde{H}$, i.e. $\g \in \tilde{H}$.  We have proven $G_N \leq \tilde{H}$, as wanted. \\

$(\supseteq)$ Let $\tilde{H} \in \{H \leq \Z^d : [\Z^d:H] < \infty \textrm{ and }H \geq G_N \textrm{ for some }N\}$.  As mentioned above we know $G_N \in \FF(\sigma_\gothG)$.  By Lemma \ref{propsofFF}, $\FF(\sigma_\gothG)$ is closed under supergroups, and thus $\tilde{H} \in \FF(\sigma_\gothG)$ as wanted. \end{proof}

Not only does the above describe the set $\FF(\sigma_\gothG)$ associated to an odometer, it can be used to characterize the conjugacy class of the odometer, as described in the next result.

%
%

\begin{cor}
\label{FFdistinguishesodoms}
Let $(X_\gothG, \sigma_\gothG)$ and $(X_\gothH, \sigma_\gothH)$ be $\Z^d$-odometers corresponding to sequences $\gothG = \{G_1, G_2, G_3, ...\}$ and $\gothH = \{H_1, H_2, H_3, ...\}$, respectively.  $(X_\gothG, \sigma_\gothG)$ and $(X_\gothH, \sigma_\gothH)$ are conjugate if and only if $\FF(\sigma_\gothG) = \FF(\sigma_\gothH)$.
\end{cor}
\begin{proof} The $(\so)$ direction is immediate, so we focus on the converse here.  Assume $\FF(\sigma_G) = \FF(\sigma_H)$  and set $i_1 = 1$.  Clearly $G_1 \in \FF(\sigma_\gothG)$ and so $G_1 \in \FF(\sigma_\gothH)$ as well.  By  Theorem \ref{FFofodoms} we then know there is $i_2$ such that $G_1 \geq H_{i_2}$.  As $H_{i_2} \in \FF(\sigma_\gothH) = \FF(\sigma_\gothG)$, we can apply Theorem \ref{FFofodoms} again to obtain $i_3> i_1$ with $H_{i_2} \geq G_{i_3}$.  Continuing in this fashion, we obtain a sequence $\gothG' = \{G_{i_1}, H_{i_2}, G_{i_3}, H_{i_4}, ...\}$ of finite-index subgroups of $\Z^d$ where
\[G_{i_1} \geq H_{i_2} \geq G_{i_3} \geq H_{i_4} \geq \cdots\]
Since inserting or deleting subgroups from a decreasing sequence of subgroups defining an odometer does not change its conjugacy class, we see $(X_{\gothG'}, \sigma_{\gothG'})$ is conjugate to both $(X_\gothG, \sigma_\gothG)$ and $(X_\gothH, \sigma_\gothH)$, and therefore $(X_\gothG, \sigma_\gothG)$ and $(X_\gothH, \sigma_\gothH)$ are conjugate to one another, as desired. 
\end{proof}

We now have a way to associate to an odometer $(X_\gothG, \sigma_\gothG)$ the set $\FF(\sigma_\gothG)$ of finite index subgroups of $\Z^d$, closed under supergroups and intersections.  We next want to show the converse: given such a set of subgroups, one can construct an associated odometer.  This will be shown by the next lemma which also gives a characterization of when the associated odometer is infinite.

%
%

\begin{lemma} 
\label{odometergeneration}
Let $\mathcal{F}$ be a collection of finite-index subgroups of $\Z^d$ that is closed under supergroups and intersections.  Then:
\begin{enumerate}
\item There is an odometer $(X_\gothG, \sigma_\gothG)$, unique up to conjugacy, such that $\F = \mathcal{FF}(\sigma_\gothG)$.
\item If $\mathcal{F}$ is infinite, then $(X_\gothG, \sigma_\gothG)$ is infinite.
\item If $\stackbin[G \in \F]{}{\bigcap} G = \{\0\}$, then $(X_\gothG, \sigma_\gothG)$ is free.
\end{enumerate}
\end{lemma}
\begin{proof} 
As there are only countably many subgroups of $\Z^d$ of finite index, $\mathcal{F}$ must be countable, and so we can write 
$\mathcal{F} = \{H_1, H_2, H_3, ...\}$.  For each $k \geq 1$, let $G_k = \cap_{j=1}^k H_j$.  If  
$\F$ is finite and has, say, $t$ members, then for every $k > t$ set $G_k = G_t$.
This gives a decreasing sequence 
$\gothG = \{G_1, G_2, G_3, ..\}$ of finite index subgroups of $\Z^d$ yielding a corresponding odometer 
$(X_\gothG, \sigma_\gothG)$. 

To show $\F = \mathcal{FF}(\sigma_\gothG)$, let $H\in \F$.  Then $H = H_i$ for some $i$ and thus $G_k\le H$ for any $k\ge i$.  From Theorem \ref{FFofodoms} it follows that 
$H \in \FF(\sigma_\gothG)$.  Next, take  $G \in \FF(\sigma_\gothG)$.
By Theorem \ref{FFofodoms}, $G$ is a supergroup of some $G_k$.  Because $\F$ is closed under intersections and
$G_k = \cap_{j=1}^k H_j$, we know $G_k \in \F$.  Using that $\F$ is also closed under supergroups, we have $G \in \F$ as well, completing the proof of the first statement.  

If $\F$ is infinite, the sequence $\{G_k\}$ is not eventually constant, so the odometer $X_\gothG$ is infinite.  
If $\stackbin[G \in \F]{}{\bigcap} G = \stackbin[k]{}{\bigcap} H_k = \{\0\}$, then $\stackbin[k]{}{\bigcap} \, G_k = \{\0\}$ as well, so $(X_\gothG, \sigma_\gothG)$ is free (see Theorem \ref{basicodomprops}).
\end{proof}

The odometer $(X_\gothG, \sigma_\gothG)$ constructed in the above result will be called the {\bf{odometer generated by }} $\F$. \\

We now turn to the connection between the set $\FF(\T)$ and the kinds of odometer actions onto which $(X,\T)$ can factor:

%
%
\begin{lemma} 
\label{infodomfactorexists}
Let $(X,\T)$ be a measure-preserving $\Z^d$-system.  Then:
\begin{enumerate}
\item $(X,\T)$ factors onto an infinite $\Z^d$-odometer if and only if $\FF(\T)$ is infinite.
\item $(X,\T)$ factors onto a free $\Z^d$-odometer if and only if $\stackbin[G \in \FF(\T)]{}{\bigcap} G = \{\0\}$.
\end{enumerate}
\end{lemma}
\begin{proof} 
We begin with statement (1).  Suppose first that $(X,\T)$ factors onto the infinite $\Z^d$-odometer $(X_\gothG, \sigma_\gothG)$ associated to the sequence $\gothG = \{G_1, G_2, G_3, ...\}$.  It follows that for every $k$, $(X,\T)$ factors onto $\Z^d/G_k$, i.e. $G_k \in \FF(\T)$.  As $\sigma_\gothG$ is infinite, there must be infinitely many different $G_k$, so $\FF(\T)$ is infinite, as wanted.

Conversely, if $\FF(\T)$ is infinite, then by Lemma \ref{odometergeneration} we can let $(X_\gothG, \sigma_\gothG)$ with 
$\gothG = \{G_1, G_2, G_3, ...\}$ be the infinite odometer generated by $\FF(\T)$.   Then  $\FF(\T) = \FF(X_\gothG)$ and thus $G_k\in \FF(\T)$ for every $k$.  By definition, that means $(X,\T)$ factors onto every $(\Z^d/G_k, \tau)$ and so by Theorem \ref{factorexists}$,(X,\T)$ must also factor onto $(X_\gothG, \sigma_\gothG)$ as wanted. \\

Now for statement (2).  Suppose first that $(X,\T)$ factors onto the free $\Z^d$-odometer $(X_\gothG, \sigma_\gothG)$ associated to the sequence $\gothG = \{G_1, G_2, G_3, ...\}$.  As before, we have $G_k \in \FF(\T)$ for every $k$, which implies $\stackbin[G \in \FF(\T)]{}{\bigcap} G \subseteq \bigcap_k G_k = \{\0\}$.

Conversely, assume $\stackbin[G \in \FF(\T)]{}{\bigcap} G = \{\0\}$.  
Use Lemma \ref{odometergeneration} to obtain a free odometer $(X_\gothG, \sigma_G)$ where $\gothG = \{G_1, G_2, G_3, ...\}$ and $\FF(\sigma_G) = \FF(\T)$.
Since each $G_k\in\FF(\sigma_G)$, we then have that  $(X,\T)$ factors onto every $(\Z^d/G_k, \tau)$.  Thus  by Theorem \ref{factorexists} $(X,\T)$ must also factor onto $(X_\gothG, \sigma_G)$, as wanted. 
\end{proof}

We can now characterize when, given a F{\o}lner rank one $\Z^d$ system, there is some infinite odometer on which it will factor.

%
%

\begin{thm}
\label{someodometerfactor}
Let $\T : \Z^d \actson (X, \mu)$ be F{\o}lner rank one.  The following are equivalent:
\begin{enumerate}
\item $(X,\mu,\T)$ factors onto some infinite odometer.
\item For all $j \in \N$, there is a finite-index subgroup $G_j$ with $[\Z^d : G_j] \geq j$ so that for all $\ep > 0$, there exists $N \in \N$ for all $n \geq m \geq N$, there is $\g \in \Z^d$ such that
\[\frac{\#(\{\i \in I_{m,n} : \i \not \equiv \g  \mod G_j\})}{\#(I_{m,n})} < \ep.\]
\end{enumerate}
\end{thm}
\begin{proof} 
(1) $\so$ (2): Assume $(X,\mu,\T)$ factors onto some infinite odometer $(X_\gothG, \sigma_\gothG)$ where $\gothG = \{G_1, G_2,...\}$.  Composing the factor map with the projection onto the $j^{th}$ coordinate, Theorem \ref{finitefactorthm} yields the result in (2).

(2) $\so$ (1):  
By Theorem \ref{finitefactorthm}, $G_j \in \FF(\T)$ for every $j$.  Since $[\Z^d : G_j] \geq j$,  $\FF(\T)$ must be  infinite, and (1)  follows from Lemma \ref{infodomfactorexists}.  
 \end{proof}

%
%

\section{Rank one $\Z^d$ systems conjugate to an odometer}

In this section we explore when a F{\o}lner rank one $\Z^d$-system not only factors onto some infinite odometer but in fact is conjugate to that odometer.  We first consider the situation of when the F{\o}lner rank one system and the odometer are given, and one wants to check if they are conjugate.  
This is followed by a theorem that considers how one can check, given the cutting and stacking procedure that generates a F{\o}lner rank one system, if there is {\it{any}} odometer to which it is conjugate.  

Both of these theorems will involve two properties which, when taken together, give necessary and sufficient conditions on a F{\o}lner rank one $\Z^d$-action to guarantee the wanted conjugacy. For instance, consider the two properties listed in Theorem \ref{conj2odom} below.  
Note that condition (2b) is familiar to us from the proceeding sections:  one can think of (2b) as saying that eventually, the F{\o}lner rank one $\Z^d$-action mimics the group structure of the odometer.  
 More specifically, it requires that for each subgroup $G_j$ in the odometer structure, eventually, most descendants of every tower in the rank one system are congruent$\mod G_j$.
Importantly, condition (2b) does not depend at all on how the descendants of the first few towers in the cutting and stacking construction behave.  In order to achieve a conjugacy between a F{\o}lner rank one action and an odometer, it is necessary that \emph{every} tower is stacked in a way that does not deviate too much from the structure of the odometer.  This is captured with condition (2a), which says that eventually, the descendants of the base of each stage $l$ tower mostly coincide with a single subset of some $\Z^d/G_k$.

%
%
\begin{thm}
\label{conj2odom}
Let $\T : \Z^d \actson (X,\mu)$ be F{\o}lner rank one and let $\sigma_\gothG : \Z^d \actson (X_\gothG, \mu_\gothG)$ be the $\Z^d$-odometer corresponding to the sequence of subgroups $\gothG = \{G_1, G_2, G_3, ...\}$.  The following are equivalent:
\begin{enumerate}
\item $(X,\T)$ is conjugate to $(X_\gothG, \sigma_\gothG)$.
\item Both of the following statements hold:
\begin{enumerate}
\item For all $l \in \N$ and all $\ep > 0$, there exists $k \in \N$ and $N \in \N$ such that for all $ m \geq N$, there is $D \subseteq \Z^d/G_k$ such that
\[
\frac{\#\left(I_{l,m} \, \triangle \, \{\i \in F_m : \i + G_k \in D\} \right)}{\#(I_{l,m})} < \ep.
\]
\item For each $j$, and for every $\ep > 0$, there exists $N \in \N$ such that for all $n \geq m \geq N$, there is $\g  \in \Z^d$ such that
\[
\frac{\#(\{\i \in I_{m,n} : \i \not \equiv \g \mod G_j\})}{\#(I_{m,n})} < \ep.
\]
\end{enumerate}
\end{enumerate}
\end{thm}
\begin{proof} 
We begin with (1) $\so$ (2).  Let $\phi : X \to X_\gothG$ give a conjugacy between the systems.  Since we can also think of $\phi$ as a factor map from $X$ to $X_\gothG$, Corollary
\ref{rankoneodometerfactor} yields (2b).  Thus it suffices to establish statement (2a).  

Because each level in a tower has the same measure, the inequality in (2a) is equivalent to 
\[
\frac{\mu \left(  \left[ \stackrel[\i\in I_{l,m}]{}{\bigcup}\T^{\i}(B_m) \right] \, \triangle  \, \left[   \stackrel[\i\in F_m, \i + G_k \in D]{}{\bigcup}  \T^{\i}(B_m)   \right] \, \right)}{\mu \left( \stackrel[\i\in I_{l,m}]{}{\bigcup}\T^{\i}(B_m)  \right) } < \ep
\]
which we can rewrite as 
\begin{equation}\label{ForThmFive}
\frac{\mu \left( B_l \, \triangle  \, \left[   \stackrel[\i\in F_m, \i + G_k \in D]{}{\bigcup}  \T^{\i}(B_m)   \right] \, \right)}{\mu \left( B_l \right) } < \ep.
\end{equation}
So fix $l \in \N$ and $\ep > 0$  and then set $\delta = \frac 14 \ep \, \mu(B_l)$.  We first find the specific $k$ for statement (2a): letting $\pi_{\tilde{k}}:X_\gothG \rightarrow \Z^d/G_{\tilde{k}}$ and 
recalling that the sets $\{\pi_{\tilde{k}}^{-1}(\g+G_{\tilde{k}}) : \g + G_{\tilde{k}} \in \Z^d/G_{\tilde{k}}, {\tilde{k}} \in \N\}$ generate the $\sigma$-algebra on $X_\gothG$, let $k$ be such that 
$\phi(B_l)$ can be approximated by the $\pi_k$-pullback of some union of cosets from $\Z^d/G_k$.  More specifically, choose $k$ so that there is a union $E$ of cosets from $\Z^d/G_k$ so that
\[  
\mu_{\gothG} \left( \, \phi(B_l) \, \triangle \, \pi_k^{-1}(E) \, \right) < \delta.
\]
Since $\phi$ is a conjugacy, we thus have, for this $k$ and this collection of cosets $E$, 
\begin{equation}\label{ForThmFiveTriangle} 
\mu \left( \, B_l \, \triangle \, \phi_k^{-1}(E) \, \right) < \delta,
\end{equation}
where we set $\phi_k = \pi_k \circ \phi$, so that $\phi_k: X \rightarrow \Z^d/G_k$.

Choose $N_1 \geq l$ so that $\mu \left( \cup_{\i\in F_{N_1}} \T^{\i}(B_{N_1}) \right)> 1-\delta$.  We then find the specific $N$ for statement (2a) by choosing $N\ge N_1$, similar to what we did in part (i) in the proof of Theorem \ref{finitefactorthm}, such that for all $m\ge N$, there exists a $\g  \in \Z^d$ such that
\[ 
\frac{\mu \left( \, B_m \, - \phi_k^{-1}(\g + G_k) \, \right)}{\mu (B_m)} < \delta.
\]
Fix $m\ge N$ and let $\g$ be as described above.  This in turn means, for each $\i\in F_m$,
\begin{equation}
\label{ForThmFiveCover}
\frac{\mu \left( \, \T^{\i}(B_m) \, - \phi_k^{-1}(\g + \i + G_k) \, \right)}{\mu \left(\T^{\i}(B_m)\right)} < \delta.
\end{equation}

Finally, we specify the set $D\subseteq \Z^d/G_k$ referred to in statement (2a) by setting $D = E - (\g + G_k)$.

With $k$ and $N$ thus chosen and our set $D$ determined, we will now turn to verifying inequality (\ref{ForThmFive}).  First, observe
\begin{align}
\label{onetotwosymdif}
& \mu \left( \, B_l \, \triangle  \, \left[   \stackrel[\i \in F_m, \i + G_k \in D]{}{\bigcup}  \T^{\i}(B_m)   \right] \, \right) \leq \\
 & \qquad \quad  \mu \left( \, B_l \, \triangle  \, \phi_k^{-1}(E)\, \right)  +  \mu \left( \,\phi_k^{-1}(E)\, \triangle  \, \left[   \stackrel[\i\in F_m, \i + G_k \in D]{}{\bigcup}  \T^{\i}(B_m)   \right] \, \right). \nonumber
\end{align}
As $\phi_k^{-1}(E)\, \triangle  \, \left[   \stackrel[\i\in F_m, \i + G_k \in D]{}{\bigcup}  \T^{\i}(B_m) \right]$ is the disjoint union of
\[
S_1 =  \phi_k^{-1}(E)\, -  \, \left[   \stackrel[\i\in F_m, \i + G_k \in D]{}{\bigcup}  \T^{\i}(B_m) \right]
\]
and 
\[
S_2 = \left[   \stackrel[\i\in F_m, \i + G_k \in D]{}{\bigcup}  \T^{\i}(B_m) \right] \, - \, \phi_k^{-1}(E),
  \]
we will next bound the measures of $S_1$ and $S_2$ separately.  \\

First, consider points in $S_1 = \phi_k^{-1}(E)\, -  \left[   \stackrel[\i\in F_m, \i + G_k \in D]{}{\bigcup} \, \T^{\i}(B_m) \right]$ and further divide the points in $S_1$ by whether or not they also lie in $B_l$.  For $x \in S_1 \cap B_l$, we see that since the towers stack and $m \geq l$, $x$ lies in the $l^{th}$-tower $\mathcal{T}_l$, so we can find $\hat{\i} \in F_m$ with $x\in \T^{\hat{\i}}(B_m)$.  As $x\in S_1 = \phi_k^{-1}(E) -  \stackrel[\i\in F_m, \i + G_k \in D]{}{\bigcup}  \T^{\i}(B_m)$, it must be that $\hat{\i} + G_k \notin D$ or equivalently, $\hat{\i} + \g + G_k \notin E$.
Despite the fact that (\ref{ForThmFiveCover}) says that most of $\T^{\hat{\i}}(B_m)$ is covered by $\phi_k^{-1}(\hat{\i}+\g+G_k)$, it must be that $x\notin \phi_k^{-1}(\hat{\i}+\g+G_k)$.   For if it were, then $\phi_k(x) = \hat{\i}+\g+G_k$, which we said earlier was not in $E$; however, we also have $x\in \phi_k^{-1}(E)$, a contradiction.  So we have $x\in \T^{\hat{\i}}(B_m)$ and  $x\notin \phi_k^{-1}(\hat{\i}+\g+G_k)$; by (\ref{ForThmFiveCover}) we know the measure of such points is less than $\delta \, \mu(\T^{\hat{\i}}(B_m))$.  Therefore
\begin{equation}
\label{inequality2}
 \mu \left( S_1 \cap B_l \right) < \delta \cdot \mu\left(\stackrel[{\hat{\i}}\in F_m]{}{\bigcup} \T^{\hat{\i}}(B_m) \right) < \delta.
\end{equation}
Turning to points in $S_1 - B_l$, we see from the definition of $S_1$ that such points lie in $\phi_k^{-1}(E)$.  So
\begin{align} \label{inequality1}
\mu \left(S_1 - B_l \right)     & \le      \mu \left(  \phi_k^{-1}(E) - B_l \right) \nonumber \\
&  \le \mu \left(  \phi_k^{-1}(E)\, \triangle \, B_l \right) \nonumber \\
& < \delta 
\end{align}
by (\ref{ForThmFiveTriangle}).  Finally, from (\ref{inequality2}) and (\ref{inequality1}), we have
\begin{equation}
\label{NewBoundOnS1}
\mu(S_1) = \mu(S_1 \cap B_l) + \mu(S_1 - B_l) < \delta + \delta = 2 \delta.
\end{equation}

Second, we consider points in $S_2 = \left[ \stackrel[\i\in F_m, \i + G_k \in D]{}{\bigcup} \, \T^{\i}(B_m) \right] \, - \, \phi_k^{-1}(E)$.  For such a point $x$, we have $\hat{\i} \in F_m$ so that 
$x\in \T^{\hat{\i}}(B_m)$ with $\hat{\i} + G_k\in D$ or equivalently, $\hat{\i} + \g + G_k\in E$.  Note that we must have $\phi_k(x)\neq \hat{\i} + \g + G_k$, because otherwise we would have  
$\phi_k(x)\in E$, contradicting the definition of $S_2$.  So we have $x\in \T^{\hat{\i}}(B_m)$ and  $x\notin \phi_k^{-1}(\hat{\i}+\g+G_k)$; by  (\ref{ForThmFiveCover}) we know the measure of such points is less than $\delta \, \mu(\T^{\hat{\i}}(B_m))$.  All together we then get 
\begin{equation}\label{inequality3}
 \mu(  S_2)< \delta \cdot \mu\left(\stackrel[{\hat{\i}}\in F_m]{}{\bigcup} \T^{\hat{\i}}(B_m) \right) < \delta.
\end{equation}

Combining (\ref{NewBoundOnS1}) and (\ref{inequality3}), we obtain
\[
\mu \left( \phi_k^{-1}(E)\, \triangle  \, \left[   \stackrel[\i\in F_m, \i + G_k \in D]{}{\bigcup} \T^{\i}(B_m)   \right]  \right) = \mu(S_1 \sqcup S_2) < 3\delta,
\]
and returning to  (\ref{onetotwosymdif}), we then have
\[
\mu \left(  B_l \, \triangle  \, \left[  \stackrel[\i\in F_m, \i + G_k \in D]{}{\bigcup}  \T^{\i}(B_m)   \right]  \right) < \mu \left( B_l \, \triangle  \, \phi_k^{-1}(E) \right) + 3\delta
\]
which, using (\ref{ForThmFiveTriangle}), gives us $\mu \left(B_l \, \triangle  \, \left[   \stackrel[\i\in F_m, \i + G_k \in D]{}{\bigcup} \T^{\i}(B_m)   \right] \right) < \delta + 3\delta = 4\delta$. Therefore
\[
\frac{\mu \left( B_l \, \triangle  \, \left[   \stackrel[\i\in F_m, \i + G_k \in D]{}{\bigcup}  \T^{\i}(B_m)   \right] \right)}{\mu \left( B_l \right) } < \frac{4\delta}{\mu \left( B_l \right) } = \ep,
\] 
establishing (\ref{ForThmFive}) as wanted. \\

 We now show (2) $\so$ (1).  From statement (2b) and Theorem \ref{finitefactorthm}, we know for that $G_k\in \mathcal{FF}(\T)$ for every $k$, so there exists a factor map 
 $\pi_k: X\rightarrow \Z^d/G_k$.  It then follows by Theorem \ref{factorexists} that $X$ factors onto $X_{\gothG}$ with the map $\psi(x) = (\psi_1(x), \psi_2(x),...)$ where 
 $\psi_k(x) = \pi_k(x) - \hat{\r}_k(x)$.  As commented on in the proof of that theorem, $(X,\T)$ being rank one, and hence ergodic, means $\hat{\r}_k(x)$ is a constant function and we can write 
  $\psi_k(x) = \pi_k(x) - \hat{\r}_k$.
  
  Our goal is to show that the factor map $\psi$ is in fact a conjugacy.  
 This will be assured if $\psi^{-1}(\mathcal{X}_\gothG) = \mathcal{X}$; since these $\sigma$-algebras are generated respectively by 
 $\{\psi_k^{-1}(\g) : k \in \N, \g \in \Z^d/G_k\}$ and $\{\T^\i(B_l) : l \in \N, \i \in F_l\}$, it suffices to show that for every $l \in \N$ and every $\ep > 0$, 
 there is $k \in \N$ and $E \subseteq \Z^d/G_k$ with $\mu\left(B_l \, \triangle \, \psi_k^{-1}(E)\right) < \ep$.
 
So let $l \in \N$ and $\ep > 0$ be fixed.  By statement (2a) we obtain $k\in \N$ and $N_1\in\N$ (without loss of generality we can assume $N_1>l$) so that whenever $m\ge N_1$, there is $D\subseteq \Z^d/G_k$ so that
 \begin{equation}
 \label{fifthineq}
\frac{\#\left(I_{l,m} \, \triangle \, \{\i \in F_m : \i + G_k \in D\}\right)}{\#(I_{l,m})} < \frac{\ep}4.
\end{equation}
Since every level in the tower $\mathcal{T}_m$ has the same measure, (\ref{fifthineq}) is equivalent to 
 \[
\frac{\mu \left(  \left[ \stackrel[\i\in I_{l,m}]{}{\bigcup}\T^{\i}(B_m) \right] \, \triangle  \, \left[  \stackrel[\i\in F_m, \i + G_k \in D]{}{\bigcup} \, \T^{\i}(B_m)   \right] \right)}{\mu \left( \stackrel[\i\in I_{l,m}]{}{\bigcup}\T^{\i}(B_m)  \right) } < \frac{\ep}4,
\]
and by the definition of $I_{l,m}$ this inequality is the same as
\begin{equation}
\label{symdiffone}
\frac{\mu \left(  B_l \, \triangle  \, \left[   \stackrel[\i\in F_m, \i + G_k \in D]{}{\bigcup} \, \T^{\i}(B_m)   \right]  \right)}{\mu \left( B_l \right) } < \frac{\ep}4.
\end{equation}

Similar to what we did in part (i) of Theorem \ref{finitefactorthm}, we can find $N_2 \geq N_1$ such that for every $m \geq N_2$, there is some $\gamma \in \Z^d/G_k$   such that 
$\mu(B_m - \psi_k^{-1}(\gamma)) < \frac{\ep}4 \mu(B_m)$.  This in turn means 
\begin{equation}\label{difftwo}
\mu(\T^{\i}(B_m) - \psi_k^{-1}(\gamma + \i + G_k)) < \frac{\ep}4\mu(B_m) \hspace{.4 in}   \forall \, \i\in F_m.
\end{equation}

We now find the required set $E \subseteq \Z^d/G_k$ by fixing
 $m\ge N_2$ so large that $\mu(\mathcal{T}_m) > 1 -\frac{\ep}4$ and letting $D$ and $\gamma$ be as specified above.  We then take $E = \gamma + D$; it remains for us to prove that
$\mu\left(B_l \, \triangle \, \psi_k^{-1}(E)\right) < \ep$.   

Observe that

\begin{align}
\label{moresymdiff}
 \mu \left( B_l \, \triangle  \, \psi^{-1}_k(E)  \right) & \leq  \mu \left( B_l \, \triangle  \, \left[   \stackrel[\i\in F_m, \i + G_k \in D]{}{\bigcup}  \T^{\i}(B_m)   \right]  \right)   \nonumber \\
& \hspace{.5 in} +  \mu \left(  \left[   \stackrel[\i\in F_m, \i + G_k \in D]{}{\bigcup} \T^{\i}(B_m)   \right] \, \triangle \, \psi^{-1}_k(E)   \right)  \nonumber \\
& = \mu \left( B_l \, \triangle  \, \left[   \stackrel[\i\in F_m, \i + G_k \in D]{}{\bigcup} \T^{\i}(B_m)   \right]  \right)   \\
& \hspace{.5 in} + \mu \left(   \left[   \stackrel[\i\in F_m, \i + G_k \in D]{}{\bigcup}  \T^{\i}(B_m)   \right]  -  \psi^{-1}_k(E)  \right)  \nonumber \\
 & \hspace{.5 in} + \mu \left(   \left[   \psi^{-1}_k(E)  -  \stackrel[\i\in F_m, \i + G_k \in D]{}{\bigcup}  \T^{\i}(B_m)   \right]  \right) . \nonumber
\end{align}

The first addend in (\ref{moresymdiff}) is, by (\ref{symdiffone}), less than $\frac{\ep}4 \, \mu(B_l) \le \frac{\ep}4$. 

For the second addend of (\ref{moresymdiff}), note that 
\[ 
\mu \left(   \left[   \stackrel[\i\in F_m, \i + G_k \in D]{}{\bigcup} \T^{\i}(B_m)   \right]  -  \psi^{-1}_k(E)  \right) = \sum_{ \i\in F_m,   \i + G_k \in D} \mu \left(  \T^{\i}(B_m) - \psi^{-1}_k(E) \right).
\]
Note that since $\i + G_k \in D$, $\gamma + \i + G_k \in E$.  So by applying (\ref{difftwo}), $\mu(\T^\i(B_m) - \psi_k^{-1}(E)) \leq \mu(\T^\i(B_m) - \psi_k^{-1}(\gamma + \i + G_k)) < \frac{\ep}4\mu(B_m)$.  This makes the second addend of (\ref{moresymdiff}) less than
\[
\sum_{\i\in F_m, \i + G_k \in D}  \frac{\ep}4\mu(B_m)\,  \le \, \frac{\ep}4.
\]

For the third addend of (\ref{moresymdiff}), we divide  $\psi^{-1}_k(E)$ into the part that is not in $\mathcal{T}_m$ and the part that is in $\mathcal{T}_m = \cup_{\i\in F_m}\T^{\i}(B_m)$.  By our choice of $m$ we know $\mu(\psi^{-1}_k(E) - \mathcal{T}_m \, ) < \frac{\ep}4$.  For the rest, note that
\[ 
\mu \left(   \left[ \psi^{-1}_k(E)\cap \mathcal{T}_m \right]  -  \stackrel[\i\in F_m, \i + G_k \in D]{}{\bigcup}  \T^{\i}(B_m) \right) \leq \sum_{\i\in F_m, \i + G_k \notin D}  \mu \left( \,    \psi^{-1}_k(D+\gamma) \cap \T^{\i}(B_m)\, \right).
\]
 Recall from  (\ref{difftwo}) that we picked $\gamma$ such that $\mu(\T^{\i}(B_m) - \psi_k^{-1}(\gamma + \i + G_k)) < \frac{\ep}4\mu(B_m)$.  Here we are only concerned with those $\i$ such that $\i + G_k \notin D$ or equivalently, $\i + \gamma + G_k \notin \gamma + D = E$.  This yields, for each such $\i$, 

\[
\mu \left( \,    \T^{\i}(B_m) \cap \psi^{-1}_k(E)  \, \right) \le \mu \left( \,  \T^{\i}(B_m) - \psi^{-1}(\gamma + \i + G_k) \,  \right) < \frac{\ep}4 \mu (B_m)
\]
and thus $\stackrel[\i\in F_m, \i + G_k \notin D]{}{\sum}  \mu \left(   \psi^{-1}_k(E) \cap \T^{\i}(B_m) \right) < \frac{\ep}4$.  Thus the third addend of (\ref{moresymdiff}) is less than $\frac \ep 2$, so (\ref{moresymdiff}) becomes
\[\mu\left(B_l \, \triangle \, \psi_k^{-1}(E)\right) < \frac \ep 4 + \frac \ep 4 + \frac \ep 2 = \ep\]
as wanted.
\end{proof}

\begin{thm}
\label{conj2someodom}
Let $\T : \Z^d \actson (X,\mu)$ be F{\o}lner rank one.  The following are equivalent:
\begin{enumerate}
\item $(X,\mu,\T)$ is conjugate to an infinite $\Z^d$-odometer.
\item For all $l \in \N$ and all $\ep > 0$, there is a finite index subgroup $G$ of $\Z^d$ such that 
\begin{enumerate}
\item there exists $N_a \in \N$ so that for all $m \geq N_a$, there exists $D \subseteq \Z^d/G$ for which
\[
\frac{\#\left(I_{l,m} \, \triangle \, \{\i \in F_m : \i + G \in D\} \right)}{\#(I_{l,m})} < \ep; {\textrm{ and}}
\]
\item for every $\eta > 0$, there is $N_b \in \N$ so that for all $n > m \geq N_b$, there exists $\g + G \in \Z^d/G$ where
\[
\frac{\#(\{\i \in I_{m,n} : \i \not \equiv \g \mod G\})}{\#(I_{m,n})} < \eta.
\]

\end{enumerate}
\end{enumerate}
\end{thm}

\begin{proof} 
We begin with (1) $\so$ (2):  Suppose $(X,\T)$ is conjugate to an infinite $\Z^d$-odometer ($X_\gothG, \sigma_\gothG)$ given by $\gothG = \{G_1, G_2, G_3, ...\}$.  Since the odometer is infinite, there must be infinitely many distinct $G_j$, all of which are elements in $\FF(\sigma_\gothG)$.  Thus $\FF(\sigma_\gothG)$ is an infinite set.
Fix $l \in \N$ and $\ep > 0$.  By condition (2a) of Theorem \ref{conj2odom}, we obtain $k_{l,\ep}$ and $N_{l,\ep} \in \N$ so that whenever $m \geq N_{l,\ep}$, there is $D \subseteq \Z^d/G_{k_{l,\ep}}$ so that
\[
\frac{\#\left(I_{l,m} \, \triangle \, \{\i \in F_m : \i + G_{k_{l,\ep}} \in D\}\right)}{\#(I_{l,m})} < \ep.
\]
This gives condition (2a) above with $N_a = N_{l,\ep}$ and $G=G_{k_{l,\ep}}$.
Now fix $\eta>0$ and use $j=k_{l,\ep}$ and $\ep = \eta$ in condition (2b) of Theorem \ref{conj2odom} to find $N_{b} \in \N$ so that for all $n\ge m \geq N_{b}$, there is $\g  \in \Z^d$ so that
\[
\frac{\#\left(\{\i \in I_{m,n} : \i \not \equiv \g \mod G_{k_{l,\ep}}\}\right)}{\#(I_{m,n})} < \eta,
\]
giving condition (2b) above. Thus we obtain (2). \\

We now turn to proving (2) $\so$ (1).  This is done in three steps: we first use the finite-index subgroups given to us by (2) 
to obtain an odometer $(X_{\gothG'}, \sigma_{\gothG'})$.  We then modify $(X_{\gothG'}, \sigma_{\gothG'})$ to obtain a 
conjugate odometer $(X_{\gothG}, \sigma_{\gothG})$ whose generating groups are more controlled.  Finally, we show this new odometer is conjugate to the original F{\o}lner rank one system $(X,\mu,\T)$ .

For the first step, for each $l \in \N$ and each $\ep > 0$, let $\mathcal{G}(l, \ep)$ be the set of all finite-index subgroups of $\Z^d$ which satisfy conditions (2a) and (2b) of this theorem.  By assumption, this is nonempty.  Then define $\F$ to be the smallest set of finite-index subgroups of $\Z^d$ that is closed under supergroups, closed under intersections, and contains every member of every $\mathcal{G}(l,\ep)$.   Note that this is equivalent to saying
\begin{equation} \label{equivFdef}
\F = \left\{G \leq \Z^d : \begin{array}{l}
\textrm{there exists }\ep_1,\dots, \ep_t \in (0,1) \textrm{ and }l_1,\dots, l_t \in \N, \\
\textrm{and for each }j \in \{1, \dots, t\} \textrm{ there exists }G_j \in \mathcal{G}(l_j, \ep_j) \\
\textrm{ so that }G \geq \bigcap_{j=1}^t G_j.
\end{array}
\right\}.
\end{equation}

By Lemma \ref{odometergeneration}, there is an odometer $(X_{\gothG'}, \sigma_{\gothG'})$  such that $\FF(\sigma_{\gothG'}) = \F$.
We will next show that $\F$ is an infinite set; Lemma \ref{odometergeneration} will then also tell us that 
$(X_{\gothG'}, \sigma_{\gothG'})$ is an infinite odometer.

We will show $\F$ is an infinite set by in fact showing that $\cup_{l,\ep}\mathcal{G}(l, \ep)$ is infinite, which we do by contradiction.  
So assume that $\cup_{l,\ep}\mathcal{G}(l, \ep)$ is finite.  Then there is $t \in \N$ so that $[\Z^d:G] \leq t$ for all $G \in \cup_{l, \ep} \mathcal{G}(l, \ep)$; without loss of generality we choose such a $t$ so that $t \geq 8$.  Now, choose $l$ so that $\mu(B_l) < \frac 1{t^2}$ and let $\ep < 1$. 
The inequality in (2a) provides us with $G$, $N_a$ and $D$ so that for all $m \geq N_a$,
\begin{align}
\frac{\#\left(I_{l,m} \, \triangle \, \{\i \in F_m : \i + G \in D\} \right)}{\#(I_{l,m})} & =  \nonumber\\
\frac{\mu\left(B_l \, \triangle \, \stackrel[\i \in F_m, \i + G \in D]{}{\bigcup} \T^\i(B_m) \right)}{\mu(B_l)} & < \, \ep.  \label{bound0}
\end{align}
Note that the collection of cosets $D$ above must be nonempty, because otherwise (\ref{bound0}) would become $1 < \ep$, which is impossible.  So $D$ must contain at least one coset in $\Z^d/G$, say $\w + G$.

Applying Lemma \ref{folnerlemma}, there exists $\tilde{N}$ such that for all $n\ge\tilde{N}$ and  every $\v \in \Z^d$,
\begin{equation}
\label{bound1}
 \frac{\#(F_{n} \cap (\v + G))}{\#(F_{n})} \geq \frac {1/2} {[\Z^d:G]} \geq \frac {1}{2t}.
\end{equation}

Choose $n\ge\textrm{ max}\{\tilde{N}, N_a\}$ so that the measure of the error set of the tower $\mathcal{T}_{n}$ is less than $ \frac 12$. 
We thus have $\mu(\mathcal{T}_{n})>\frac{1}{2}$ which means $\mu(B_{n}) > \frac{1}{2} \frac{1}{  \#(F_{n}) }$.
This gives us
\begin{align*}
\mu\left( \stackrel[\i \in F_{n}, \i + G \in D]{}{\bigcup} \T^i(B_{n})\right)
& = \#(\i \in F_{n}, \i + G \in D) \cdot \mu(B_{n}) & \phantom{abc} \\
& > \#(\i \in F_{n}, \i + G \in D) \cdot \frac{1}{2}\,  \frac{1}{ \#(F_{n_k})} \\
& \geq \#(\i \in F_{n}, \i \in \w + G) \cdot \frac{1}{2}\,  \frac{1}{ \#(F_{n_k})} & \textrm{(since $\w + G \in D$)} \\
& \geq \left[\frac 1{2t} \, \#(F_{n})\right]\cdot \frac{1}{2}\,  \frac{1}{ \#(F_{n_k})} & \textrm{(from (\ref{bound1}))} \\
& = \frac {1}{4t}.
\end{align*}
Thus when $m=n$, the set in the numerator of (\ref{bound0}) is the symmetric difference of a set of measure less than $\frac 1{t^2}$ and a set of measure at least $\frac 1{4t}$.  The smallest possible measure of the symmetric difference of these two sets is when the small set is a subset of the larger one, meaning
\[
\mu\left(B_l \, \triangle \, \stackrel[\i \in F_{m}, \i + G \in D]{}{\bigcup} \T^\i(B_{m}) \right) > \frac {1}{4t} - \frac 1{t^2}.
\]
Therefore
\[
\frac{\mu\left(B_l \, \triangle \, \stackrel[\i \in F_{m}, \i + G \in D]{}{\bigcup} \T^\i(B_{m}) \right)}
{\mu(B_l)} > \frac{\frac {1}{4t} - \frac 1{t^2}}{\frac 1{t^2}} = \frac t4-1 \geq \frac 84 - 1 = 1.
\]
and this cannot be less than $\ep$ when $\ep < 1$, violating (2a).  Thus $\cup_{l, \ep} \mathcal{G}(l, \ep)$, and the set $\F$, is infinite as wanted, and we have obtained an infinite odometer $(X_{\gothG'}, \sigma_{\gothG'})$. Write $\gothG' = \{G_1', G_2', G_3', ...\}$.

\vspace{.1in}

In the final two steps of the proof, we show  $(X,\T)$ is conjugate to $(X_{\gothG'}, \sigma_{\gothG'})$ by checking that (2a) and (2b) of Theorem \ref{conj2odom} hold, but not for the sequence $\gothG'$; rather, we will check these properties for a conjugate version of $(X_{\gothG'}, \sigma_{\gothG'})$ coming from a different sequence of subgroups we now describe.

For each $l$ and $j \in \N$, choose a particular $G_{l,\frac{1}{j}} \in \mathcal{G}(l, \frac{1}{j})$ and then set, for each $k \in \N$,
\[
G_k = G_k' \cap \left(\bigcap_{i=1}^k \bigcap_{j=1}^k G_{i, \frac 1j}\right).
\]
Consider the odometer $(X_{\gothG}, \sigma_{\gothG})$ determined by the sequence $\gothG = \{G_1, G_2, G_3, ...\}$.  We claim that $(X_{\gothG}, \sigma_{\gothG})$ is conjugate to $(X_{\gothG'}, \sigma_{\gothG'})$, and to verify this we will show $\F = \FF(\sigma_{\gothG'})$ coincides with $\FF(\sigma_{\gothG})$: the claim then follows from Corollary \ref{FFdistinguishesodoms}.  

Toward that end, let $\tilde{G}\in \FF(\sigma_{\gothG'})$.  By Theorem \ref{FFofodoms} this means that $[\Z^d : \tilde{G}] < \infty$ 
and $\tilde{G} \geq G_N'$ for some $N$.  But $G_N' \geq G_N$ and thus $\tilde{G} \geq G_N$ and, again by Theorem \ref{FFofodoms}, this means $\tilde{G}\in \FF(\sigma_{\gothG})$ and so $\FF(\sigma_{\gothG'})\subseteq  \FF(\sigma_{\gothG}) $.  Now let $G\in \FF(\sigma_{\gothG})$.  Again, this means 
$[\Z^d : {G}] < \infty$ and ${G} \geq G_N$ for some $N$.    We will show $G\in \F=\FF(\sigma_{\gothG'})$, completing this step in the proof.  
Note that every $G_{l,\ep}\in  \mathcal{G}(l, \ep)  $ lies in $\F$ by definition.  Since every $G_N'$ lies in $ \FF(\sigma_{\gothG'})$ and $ \FF(\sigma_{\gothG'})=\F$, we have every $G_N' \in \F$, too.  Together these tells us that, since $\F$ is closed under intersections, $G_N\in\F$.  But $\F$ is also closed under supersets, so we have $G\in\F$, as needed.

\vspace{.1in}

For the final step in the proof, we will show that  $(X,\T)$ is conjugate to $(X_{\gothG}, \sigma_{\gothG})$ by checking that 
(2a) and (2b) of Theorem \ref{conj2odom} hold.  

Fix $l\in\mathbb{N}$ and $\ep>0$.  Recall that we choose $G_{l,\frac{1}{k}}\in \mathcal{G}(l, \frac 1k)$ to satisfy (2a) and (2b) of this theorem.  Thus 
 there exists $N_a$ (which we can assume is $>l$) where for all $m\ge N_a$, there exists 
$\widetilde{D} \subseteq \Z^d/G_{l,\frac 1k}$ so that
\begin{equation}
\label{getting2a}
 \frac{\#\left( I_{l,m} \, \triangle \, \{\i \in F_m : \i + G_{l, \frac 1k} \in \widetilde{D}\} \right)}{\#(I_{l,m})} < \frac 1k.
\end{equation}

Fix $m\ge N_a$.  As $G_k \leq G_{l, \frac 1k}$, there is a natural quotient map $q : \Z^d/G_k \to \Z^d/G_{l, \frac 1k}$. Set $D = q^{-1}(\widetilde{D})$.  Then $D\subseteq \Z^d/G_k$ and if $\i \in F_m$ and $\i + G_k \in D$, then $\i + G_{l, \frac 1k} \in \widetilde{D}$. 
Thus 
\[
\frac{\#\left(I_{l,m} \, \triangle \, \{\i \in F_m : \i + G_k \in D\} \right)}{\#(I_{l,m})} < \frac 1k < \ep,
\]
giving us (2a) of Theorem \ref{conj2odom}.

To get (2b) of Theorem \ref{conj2odom}, recall that $\F = \FF(\sigma_{\mathcal{G}'})$ where $\mathcal{G}' = \{G_1', G_2',...\}$.  
Thus for each $j$, $G_j'\in\F$.  By the description of $\F$ given in (\ref{equivFdef}), this means 
$G_j' \geq \stackrel[i=1]{t}{\cap} G_{l_i,\ep_i}$, where $G_{l_i,\ep_i} \in \mathcal{G}(l_i, \ep_i)$ for some $l_1, ...,l_t$ and $\ep_1,...,\ep_t$.  By definition of $\mathcal{G}(l_i, \ep_i)$, we know $G_{l_i,\ep_i}$ satisfies (2b) of this theorem for $l_i$ and $\ep_i$.
By Theorem \ref{finitefactorthm}, this says $G_{l_i,\ep_i}\in \FF(T)$.
Since $\FF(T)$ is closed under intersections, we have $ \stackrel[i=1]{t}{\cap} G_{l_i,\ep_i} \in \FF(T)$, 
and since it is also closed under supergroups, we get $G_j'\in \FF(T)$.  
By similar reasoning and its definition, we then get $G_j\in\FF(T)$ for each $j$.  
By the definition of $\FF(T)$, this says $(X,\mu, T)$ factors onto $(\Z^d/G_j,\delta , \tau)$ for each $j$
 and thus by Theorem \ref{factorexists}, 
$(X, T)$ factors onto $(X_{\gothG}, \sigma_{\gothG})$.  
Corollary \ref{rankoneodometerfactor} then yields (2b) of 
Theorem \ref{conj2odom} as wanted.

\end{proof}

%
%

\section{Example}

In the example of Theorem \ref{examplenottotallyergodic} we indicated how the two-dimensional case is different and more complicated than the one-dimensional case.  We continue with this type of investigation here and  show that a $\Z^2$-action does not automatically inherit the factoring properties of its subactions: 

\begin{thm} \label{examplesubactions} There is a F\/{o}lner rank one action $\T : \Z^2 \actson [0,1]$ so that $\T^{\e_1}$ and 
$\T^{\e_2}$ each factor onto the dyadic $\Z$-odometer, but $\T$ does not factor onto any free $\Z^2$-odometer. 
\end{thm}
\begin{proof} 

This action $\T$ will be constructed inductively by cutting and stacking, as described in \cite{PR} and \cite{JS}.  We thus will define a sequence of towers $\mathcal{T}_n$ and an action $\T_n$ on each tower such that in the limit we will have a $\Z^2$-action $\T$ on a space of full measure.

Begin by cutting four disjoint subintervals from $[0,1]$ of equal length $\delta_1< \frac{1}{4}$, a number that will be determined later. Associate these subintervals to a $2\times 2$ grid and label the subintervals accordingly by $\{L^1_\v : \v \in \Z^2 \cap \left([0,2) \times [0,2)\right)\}$; we think of $\v$ as recording the position of the subinterval in a tower $\mathcal{T}_1$.  Let $E_1$ be the remaining portion of $[0,1]$: we consider this the spacer portion for this step of the construction.

We then define, for $j \in \{1,2\}$, an action $\T^{\e_j}_1 : L^1_\v \to L^1_{\v + \e_j}$ to be a linear map for any $\v$ where both $\v$ and $\v + \e_j$ lie in $\Z^2 \cap \left([0,2) \times [0,2)\right)$.  In other words, the functions $\T^{\e_j}_1$ map  the interval at position $\v$ to the interval at position $\v + \e_j$ by simple translation.   Observe that $\T^{\e_1} \circ \T^{\e_2} = \T^{\e_2} \circ \T^{\e_1}$ where these maps are both defined, so we obtain a tower $\mathcal{T}_1$ of shape $\Z^2 \cap \left([0,2) \times [0,2)\right)$ and a $\Z^2$-action $\T_1$ defined on part of $\mathcal{T}_1$.  Our final $\Z^2$-action $\T$ will coincide with $\T_1$ on that portion of $[0,1]$ where $\T_1$ is defined.

For the inductive step, suppose we have a tower $\mathcal{T}_n$, which consists of $2^{4^{n-1}}\times 2^n$ subintervals of equal length 
$\delta_n$, and labeled by $\{L^n_\v : \v \in \Z^2 \cap \left([0,2^{4^{n-1}}) \times [0,2^n)\right)\}$.  Also suppose that we have maps
$\T^{\e_j}_n : L^n_\v \to L^n_{\v + \e_j}$ that translate subinterval $L^n_\v$ to subinterval $L^n_{\v + \e_j}$ for any $\v$ where both $\v$ and $\v + \e_j$ lie in $\Z^2 \cap \left([0,2^{4^{n-1}}) \times [0,2^n)\right)$.  Let $E_n$ be the portion of $[0,1]$ not in $\mathcal{T}_n$.

To construct the tower $\mathcal{T}_{n+1}$ from $\mathcal{T}_n$, we first cut each interval $L^n_\v$ into $ \left( 2^{4^{n} - 4^{n-1}} - 1 \right) \times 2$  equal size subintervals so that we create $ \left( 2^{4^{n} - 4^{n-1}} - 1 \right) \times 2$ copies of $\mathcal{T}_n$ which we then place into two slightly staggered rows.  This is shown in Figure \ref{fig:figure1}, where $t_n =  \left( 2^{4^{n} - 4^{n-1}} - 1\right)$ and $\mathcal{T}_n^i$ indicates the various copies of $\mathcal{T}_n$.  The length of each subinterval is thus 
$\delta_{n+1} = \frac{\delta_n}{\left( 2^{4^{n} - 4^{n-1}} - 1 \right) \times 2}$: we cut from $E_n$ enough subintervals of this same length to result in a tower $\mathcal{T}_{n+1}$ of size $2^{4^n}\times 2^{n+1}$, filling these new subintervals into the locations indicated by the starred areas in Figure \ref{fig:figure1}.  Each of these subintervals is then labeled by $\{L^{n+1}_\v : \v \in \Z^2 \cap \left([0,2^{4^n}) \times [0,2^{n+1})\right)\}$ to indicate its location in $\mathcal{T}_{n+1}$ and we let $E_{n+1}\subset E_n$ be the remaining portion of $[0,1]$.

 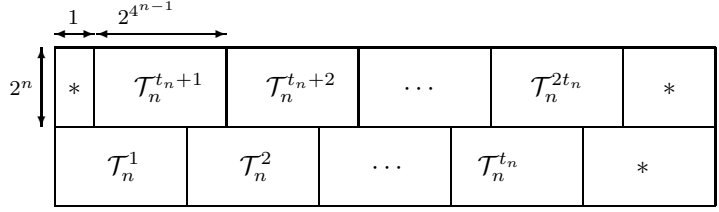
\begin{figure}
 \begin{center}
 \begin{picture}(250,80)
 \multiput(0,0)(0,30){3}{\line(1,0){250}}
 \multiput(0,0)(50,0){6}{\line(0,1){30}}
 \multiput(15,30)(50,0){5}{\line(0,1){30}}
 \multiput(0,30)(250,0){2}{\line(0,1){30}}
 \put(20,12){$\mathcal{T}_n^1$}
  \put(70,12){$\mathcal{T}_n^2$}
  \put(120,12){$\cdots$}
    \put(160,12){$\mathcal{T}_n^{t_n}$}
     \put(30,42){$\mathcal{T}_n^{t_n+1}$}
  \put(80,42){$\mathcal{T}_n^{t_n+2}$}
    \put(132,42){$\cdots$}
    \put(180,42){$\mathcal{T}_n^{2t_n}$}
    \put(10,65){\vector(-1,0){10}}
        \put(10,65){\vector(1,0){5}}
        \put(5,69){\footnotesize{$1$}}
        \put(5,42){$*$}
                \put(230,42){$*$}
                        \put(220,12){$*$}
                        \put(40,65){\vector(-1,0){24}}
                        \put(40,65){\vector(1,0){25}}
                        \put(24,69){\footnotesize{$2^{4^{n-1}}$}}
                        \put(-5,40){\vector(0,-1){10}}
                        \put(-5,40){\vector(0,1){20}}
                        \put(-17,42){\footnotesize{$ 2^n$}}
 \end{picture}
 \end{center}
 \caption{Construction of tower $\mathcal{T}_{n+1}$ from tower $\mathcal{T}_n$.}
   \label{fig:figure1}
 \end{figure}
 
We then define, for $j \in \{1,2\}$, a map $\T^{\e_j}_{n+1} : L^{n+1}_\v \to L^{n+1}_{\v + \e_j}$ by simple translation, for any $\v$ where both $\v$ and $\v + \e_j$ lie in $\Z^2 \cap \left( \,  [0,2^{4^n})    \times   [0,2^{n+1}) \,   \right)$.   Observe that $\T^{\e_j}_{n+1}$ coincides with $\T^{\e_j}_{n}$ where these both are defined, and that 
  $\T^{\e_1} \circ \T^{\e_2} = \T^{\e_2} \circ \T^{\e_1}$ where both are defined. We thus obtain a tower $\mathcal{T}_{n+1}$ of shape $\Z^2 \cap \left([0,2^{4^n}) \times [0,2^{n+1})\right)$ and a $\Z^2$-action $\T_{n+1}$ defined on an increasingly larger portion of $[0,1]$.

The total of the lengths of all the new subintervals introduced into $\mathcal{T}_{n+1}$ is bounded above by $\frac{1}{2^n}\delta_1$, and thus the sum of the lengths of all the intervals used in this construction is bounded by $4\delta_1 + \sum \frac{1}{2^n}\delta_1$: we can thus choose $\delta_1$ so that the sum of the lengths of all the intervals used in the construction is exactly 1.  We then define 
$\T : \Z^2 \actson [0,1]$ by setting, for a.e. $x\in [0,1]$, $\T^{\e_j}(x) = \T^{\e_j}_{n}(x)$ where $n$ is large enough so that $x\in L^{n}_\v$ and both ${\v}$ and ${\v + \e_j}$ lie in $\Z^2 \cap \left([0,2^{4^{n-1}}) \times [0,2^n)\right)$.  We have thus defined a F{\o}lner rank one action $\T: \Z^2 \actson [0,1]$.

\vspace{.1in}


We now want to show that this $\Z^2$-action has the desired properties, and we will begin by showing it
does {\it{not}} factor onto any free odometer.  Suppose that $(X,\T)$ factors onto some $\Z^d$-odometer $(X_\gothG,\sigma_\gothG)$ given by the sequence of subgroups $\gothG = \{G_j\}$.   We will apply Corollary \ref{rankoneodometerfactor} to show that eventually, most descendants \emph{cannot} be congruent mod $G_j$ unless $(1,0) \in G_j$.  To do this, first fix $j$ and $\epsilon < \frac{1}{4}$.  Next, for any $m\in\N,$ look at $I_{m,m+1}$ and note that, by the construction (as visualized by Figure \ref{fig:figure1}),
\begin{align*}
I_{m,m+1} & =  \left\{(k \cdot 2^{4^{m-1}}, 0) : k \in \{0, ..., t_m-1\}\right\}\\
& \qquad \bigsqcup \left\{(1 + k \cdot 2^{4^{m-1}}, 2^m) : k \in \{0, ..., t_m-1\}\right\}.
\end{align*}
This means we can organize the elements in $I_{m,m+1}$ into pairs as each $\w$  in the bottom row can be paired with $\w + (1, 2^m)$ in the top row. But the two elements of each pair are not congruent$ \mod G_j$ unless $(1, 2^m) \in G_j$.  Thus  it is impossible for there to exist one single $\g \in \Z^d$ so that $\#(\i \in I_{m,m+1} : \i \not \equiv \g \mod G_j)$ is less than a quarter of $\#(I_{m,m+1})$ unless $(1,2^m) \in G_j$.

At the same time, we can also think of the elements in $I_{m,m+2}$ as divided into four rows, at heights $0$, $2^m$, $2\cdot 2^m$ and $3\cdot2^m$.  The elements in the middle two rows can be organized into pairs, as each $\w$ in the row at height $2^m$ can be paired with $\w + (0,2^m)$ in the row at height $2 \cdot 2^m$.  The two elements of each of these pairs are not congruent$\mod G_j$ unless $(0,2^m) \in G_j$, and so as before there cannot exist a particular $\g \in \Z^d$ so that $\#(\i \in I_{m,m+1} : \i \not \equiv \g \mod G_j)$ is less than a quarter of $\#(I_{m,m+1})$ unless $(0,2^m) \in G_j$.

We have $(1,2^m) \in G_j$ and $(0,2^m) \in G_j$.  As $G_j$ is a group, we can conclude $(1,0) \in G_j$.  This holds for every $j$, meaning $(1,0) \in \bigcap_j G_j$, so by Theorem \ref {basicodomprops} this odometer $(X_\gothG, \sigma_\gothG)$ onto which $(X,\T)$ factors cannot be free.


\vspace{.1in}

We will finish this example by showing that for $j \in \{1,2\}$, $\T^{\e_j} : \Z \actson X$ {\it{does}} factor onto the dyadic $\Z$-odometer, i.e. the action $\sigma_\gothG : \Z \actson X_\gothG$ coming from the sequence $\gothG = \{G_1, G_2, ...\}$ of subgroups of $\Z$, where $G_k = 2^k\Z$.  Note that $\T^{\e_j}$ is not rank one ($\T^{\e_1}$ is not ergodic, for example), so we cannot apply Corollary \ref{rankoneodometerfactor}.  Rather, we will establish the existence of appropriate factor maps by applying Theorem \ref{factorexists}. In other words, we will show that for every $k$, $(X, \mu, \T^{\e_j})$ factors onto the finite action $(\Z/2^k\Z, \delta_k, \tau_k)$.

\vspace{.1in}


We first consider the horizontal direction $\T^{\e_1}$.  Fix $k$; choose $N_0$ so that $4^{N_0-1} \geq k$; this ensures $ 2^{4^{N_0-1}} \geq 2^k$ (i.e. the horizontal length of tower $\mathcal{T}_{N_0}$ is larger than $2^k$) and furthermore, that $2^k$ divides $ 2^{4^{N_0-1}}$.  For notational convenience, we will set $h_0 = 2^{4^{N_0-1}}$.
For $x\in  \mathcal{T}_{N_0}$, we have $x\in L^{N_0}_\v$ for some $\v= (v_1, v_2)$.  We can then define 
$\pi(x) = v_1 + 2^k\Z$.  For all other $x\in X$, we will make use of its eventual appearance in some $\mathcal{T}_n$ and its location in
this $\mathcal{T}_n$ in relation to the tower  $\mathcal{T}_{N_0}$ in order to define $\pi(x)$.  This will be done by induction.

Suppose $x\in \mathcal{T}_{N_0+1} - \mathcal{T}_{N_0}$.  Then $x$ lies in one of three regions of new spacers in 
$\mathcal{T}_{N_0+1}$ (see Figure \ref{fig:figure1}).  
\begin{itemize}
\item The first possibility is that $x$ is in the bottom-right corner of the tower $\mathcal{T}_{N_0+1}$; in this case 
$\T^{-h_0\e_1}(x)$ lies in $\mathcal{T}_{N_0}$ and we can thus define $\pi(x) = \pi( \T^{-h_0\e_1}(x) )$.

\vspace{.1in}

\item The second possibility is that $x$ is in the top-left part of $\mathcal{T}_{N_0+1}$; in this case, $\T^{h_0\e_1}(x)$ is in the tower 
$\mathcal{T}_{N_0}$, and we can set  $\pi(x) = \pi(\T^{h_0\e_1}(x) )$.

\vspace{.1in}

\item The last possibility is when $x$ is in the top-right portion of $\mathcal{T}_{N_0+1}$, in which $ \T^{-h_0\e_1}(x)$ lies in 
$\mathcal{T}_{N_0}$ and we can define $\pi(x) = \pi(  \T^{-h_0\e_1}(x) )$.

\end{itemize}
This defines $\pi$ on all $x\in  \mathcal{T}_{N_0+1}$ and we can similarly define $\pi$ on 
$ \mathcal{T}_{N_0+2} -  \mathcal{T}_{N_0+1}$ by where the points on the new spacers lie in relationship to the tower 
$\mathcal{T}_{N_0+1}$, which in turn relies on its location in relation to the tower 
$\mathcal{T}_{N_0}$.   For instance, if $x$ is in the bottom-right corner of the tower $ \mathcal{T}_{N_0+2}$, then 
$\T^{-h_1\e_1}(x)$ lies in $\mathcal{T}_{N_0+1}$, where $h_1$ is the horizontal length of tower  $\mathcal{T}_{N_1}$.  We can thus define  $\pi(x) = \pi( \T^{-h_1\e_1}(x) )$, and similarly for the other new points in $\mathcal{T}_{N_0+2}$.
Continuing in this vein, we will eventually have $\pi$ defined on almost all of $X$.

This yields a measurable map $\pi : X \to \Z/2^k\Z$.  We need to show $\pi$ is a factor map, i.e. that $\pi(\T^{\e_1}x)= \tau_k( \pi(x) )$ where $\tau_k(x) = x+1$ mod $2^k\Z$.  This will also be done by induction.

We begin with the situation where $x, \T^{\e_1}x \in \mathcal{T}_{N_0}$, and let $\v= (v_1, v_2)$ be such that  $x\in L^{N_0}_\v$.  There are two cases: 
\begin{description}
\item[Case 1] The point $x$ does not sit on the right edge of the tower $\mathcal{T}_{N_0}$, so $v_1\neq h_0-1$.  In this case, 
$\T^{\e_1}x \in L^{N_0}_{\v+ \e_1}$ and we have
\begin{align*}
\pi(\T^{\e_1}x)  = (v_1 + 1) + 2^k\Z  & = (v_1 + 2^k\Z) + (1+2^k\Z) \\
& = \tau_k(v_1 + 2^k\Z)  = \tau_k(\pi(x)).
\end{align*}
\item[Case 2] The point $x$ does sit on the right edge of the tower $\mathcal{T}_{N_0}$, so $v_1 = h_0-1$.  In this case, 
$\T^{\e_1}x \in L^{N_0}_{\w}$ where $w_1=0$, and, using that $2^k$ divides $h_0= 2^{4^{N_0-1}}$, we thus have
\begin{align*}
\pi(\T^{\e_1}x)  = 0 + 2^k\Z  = h_0 + 2^k\Z &  = (h_0-1 + 2^k\Z) + (1+2^k\Z) \\
& = \tau_k(h_0-1 + 2^k\Z) = \tau_k(\pi(x)).
\end{align*}
\end{description}

 Next, consider the situation where $x, \T^{\e_1}x \in \mathcal{T}_{N_0+1}$.  Where these two points lie in $\mathcal{T}_{N_0+1}$ determine how $\pi$ is defined on them, so we consider the following list of possibilities:
\begin{description}
\item[Case 1]  If $x, \T^{\e_1}x \in \mathcal{T}_{N_0}$, then by prior work we know $\pi(\T^{\e_1}x) = \tau_k(\pi(x))$.
\item[Case 2] Suppose $x  \in \mathcal{T}_{N_0}$ but $\T^{\e_1}x \notin \mathcal{T}_{N_0}$. Then by construction, $\T^{\e_1}x$ must lie in the bottom or top right section of spacers indicated in Figure \ref{fig:figure1} and in fact must be in the first column of one of those sections. Then, similarly to the previous Case 2, 
\begin{align*}
\pi(\T^{\e_1}x)  = \pi( \T^{-h_0\e_1}(\T^{\e_1}x))  & = 0 + 2^k\Z   = (h_0-1 + 2^k\Z) + (1+2^k\Z) \\
& = \tau_k(h_0-1 + 2^k\Z)   = \tau_k(\pi(x)).
\end{align*}
\item[Case 3] Suppose $\T^{\e_1}x \in \mathcal{T}_{N_0}$ but $x\notin \mathcal{T}_{N_0}$. Then by construction, this means $x$ lies in the upper left portion of spacers in Figure \ref{fig:figure1} and $\T^{\e_1}x$ must lie in the first column of $ \mathcal{T}_{N_0}$.  We thus have
\begin{align*}
\pi(\T^{\e_1}x)   = 0 + 2^k\Z   & =  (h_0-1 +2^k\Z)  + (1+2^k\Z) \\
& = \tau_k(h_0-1 + 2^k\Z) = \tau_k( \pi(\T^{h_0\e_1}(x) )) = \,\,\,  \tau_k(\pi(x)).
\end{align*}
\item[Case 4] Suppose both $x, \T^{\e_1}x \in \mathcal{T}_{N_0+1} - \mathcal{T}_{N_0}$.  If they both lie in the bottom or top right sections of spacers from Figure \ref{fig:figure1}, then $\T^{-h_0\e_1}(x)$ and $\T^{-h_0\e_1}(\T^{\e_1}x)$ will maintain their relative positions and 
 thus the definition of $\pi$ will yield the result in a straightforward manner.  Less obvious is the case where $x$ lies in the 
 upper right region while $\T^{\e_1}x$ lies in the upper left region.  But because the upper right region has horizontal length 
 $h_0-1$, we then have that $\T^{-h_0\e_1}(x) \in L^{N_0}_\v$ with $v_1 =  h_0-2$.  Meanwhile, having $\T^{\e_1}x$ lie in the upper 
 left region means $\T^{h_0\e_1}(\T^{\e_1}x) \in L^{N_0}_\v$ with $v_1 = h_0-1$.  We thus get 
 \begin{align*}
\pi(\T^{\e_1}x)  = \pi( \T^{h_0\e_1}(\T^{\e_1}x) ) & =   h_0-1 +2^k\Z  \\
  & =  (h_0-2 +2^k\Z) +   (1+2^k\Z)  \\
  & = \tau_k(h_0-2 + 2^k\Z)  \\
  & = \tau_k( \pi(\T^{-h_0\e_1}(x)  )) =  \tau_k(\pi(x)).
\end{align*}
\end{description}

We now have that for all points with  $x, \T^{\e_1}x \in \mathcal{T}_{N_0+1}$,  $\pi( \T^{\e_1}x ) = \tau_k( \pi(x) )$.  Note that we did this by considering the points associated to $x$ and $\T^{\e_1}x$ by the definition of  $\pi$ (e.g. $\T^{-h_0\e_1}(x)$, etc.).  These associated points are in $\mathcal{T}_{N_0}$, allowing us to apply that we had already showed that $\pi \circ \T^{\e_1} = \tau_k \circ \pi$ holds on $\mathcal{T}_{N_0}$.  The exact same reasoning can be used for induction; if we assume that for all $x, \T^{\e_1}x \in \mathcal{T}_{m}$, $\pi( \T^{\e_1}x ) = \tau_k( \pi(x) )$, we can then show that it also holds for all $x, \T^{\e_1}x \in \mathcal{T}_{m+1}$.
Having verified that $\pi$ intertwines $\T^{\e_1}$ with $\tau$,  Theorem \ref{factorexists} yields that $\T^{\e_1}$ factors onto the dyadic odometer.

\vspace{.1in}


We finish by considering the vertical direction $\T^{\e_2}$.  Fix $k$; similar to before we seek a factor map $\pi$ from $(X,\T^{\e_2})$ to $(\Z/2^k\Z, \tau_k)$.  First recall from the construction that when a tower $\mathcal{T}_n$ is cut into $t_n$ copies and positioned as in Figure 1 to create tower $\mathcal{T}_{n+1}$, any point $x\in L^n_{(v_1,v_2)}$ is moved to a level 
$L^{n+1}_{(\widetilde{v}_1,\widetilde{v}_2)}$ where $\widetilde{v}_2$ is either $v_2$ or $v_2+2^n$.  We can then unambigously define $\pi$ on all 
$\mathcal{T}_n$, $n\ge k$, by setting $\pi(x) = v_2 + 2^k\Z$ where $x\in L^n_{(v_1,v_2)}$.

To show that $\pi$ intertwines the maps $\T^{\e_2}$ and $\tau_k$, note that for almost every $x$, there exists some $n$ such that both $x, \T^{\e_2}x \in \mathcal{T}_n$.  We then have that $x\in  L^n_{(v_1,v_2)}$ and thus $\T^{\e_2}x \in L^n_{(v_1,v_2+1)}$ .
Then $\pi(\T^{\e_2}x ) = v_2+1 + 2^k\Z = (v_2 + 2^k\Z)   +(1 + 2^k\Z) = \tau_k(v_2 + 2^k\Z ) = \tau_k(\pi(x))$, as wanted. 
By Theorem \ref{factorexists}, $(X,\T^{\e_2})$ factors onto the dyadic odometer. 
\end{proof} 

This example gives a specific $\Z^2$-action that does not factor onto a free odometer, despite the fact that each $(X,\T^{\e_j})$ {\it{does}} factor onto a free odometer, and indicates how the $\Z^2$-action carries complexity beyond what can be seen by the subactions.

\bibliographystyle{amsplain}

\end{document}